\pdfoutput=1
\documentclass[11pt]{article}

\usepackage[margin=1.1in]{geometry}
\usepackage{amsmath,amssymb,amsthm,mathtools}
\usepackage{xcolor}
\usepackage{tikz}
\usetikzlibrary{arrows.meta,decorations.markings,decorations.pathreplacing,patterns}
\usepackage{array}
\usepackage{booktabs}
\usepackage[expansion=false]{microtype}
\usepackage{enumitem}
\usepackage[hidelinks]{hyperref}
\hypersetup{
  pdftitle={Which Holonomy Signatures Are Realizable? A Complete Answer for Closed
            Surfaces},
  pdfauthor={Leyi Zhu},
  pdfsubject={Seamless parametrization; holonomy signatures; k-differentials},
  pdfkeywords={seamless parametrization, holonomy signature, cross field,
               quadrangulation, quad mesh, k-differentials, strata,
               mapping class group, flat cone metric}}

\newtheorem{theorem}{Theorem}[section]
\newtheorem{lemma}[theorem]{Lemma}
\newtheorem{proposition}[theorem]{Proposition}
\newtheorem{corollary}[theorem]{Corollary}

\newtheorem{algorithm}[theorem]{Algorithm}
\theoremstyle{definition}
\newtheorem{definition}[theorem]{Definition}

\theoremstyle{remark}
\newtheorem{remark}[theorem]{Remark}

\newcommand{\ZZ}{\mathbb{Z}}
\newcommand{\CC}{\mathbb{C}}
\newcommand{\RR}{\mathbb{R}}

\newcommand{\PP}{\mathbb{P}}
\newcommand{\Zf}{\ZZ_4}
\newcommand{\im}{\operatorname{im}}
\newcommand{\Sp}{\operatorname{Sp}}
\newcommand{\MCG}{\operatorname{MCG}}
\newcommand{\divisor}{\operatorname{div}}

\newcommand{\content}{\operatorname{cont}}
\newcommand{\gen}[1]{\langle #1 \rangle}

\title{Which Holonomy Signatures Are Realizable?\\
\large A Complete Answer for Closed Surfaces}

\author{Leyi Zhu\\[2pt]
\normalsize Courant Institute of Mathematical Sciences\\
\normalsize New York University}

\date{}

\begin{document}

\maketitle

\begin{abstract}
A seamless parametrization of a closed oriented surface carries a discrete invariant, its
\emph{holonomy signature}: the cone angles, all multiples of $\pi/2$, together with the
rotational holonomy $\rho\colon H_1(M\setminus C)\to\Zf$ of the induced cross field. This is
the datum a quadrangulation prescribes. It does not determine the parametrization, since
structures sharing a signature come in positive-dimensional families, but it does decide
whether any parametrization exists at all. Shen, Zhu, Capouellez, Panozzo, Campen and Zorin
asked which signatures occur and gave a sufficient condition of gcd type. Which signatures
are realizable has remained open, and the reason is visible in the shape of what was known:
the sufficient conditions are combinatorial, while the one obstruction, a torus with cones
of angles $3\pi/2$ and $5\pi/2$, is conformal.

We answer the question. A Reduction Lemma shows that the mapping class group acts on
signatures with fixed cone angles with orbits classified by the subgroup $\im\rho\le\Zf$
alone, so at most three cases survive per angle multiset instead of $4^{2g}$. A dictionary
then identifies seamless parametrizations with meromorphic $4$-differentials, under which
$\im\rho$ measures primitivity, and realizability becomes non-emptiness of a stratum of
primitive $k$-differentials with $k=4/d$ and $\im\rho=\gen{d}$. Unwinding this against the
known classification of such strata leaves exactly five exceptional families; every other
admissible signature is realizable, in every genus. Two of the five appear to be new, and
both live in genus two. Four of the five lie outside the gcd condition, and
\S\ref{sec:examples} settles the whole region that condition leaves open.

The non-emptiness half is also made constructive. We exhibit an explicit one-vertex
square-tiled surface in every genus, together with a local surgery that cuts a saddle
connection from a cone to itself and glues in a single square, splitting one cone into two
of prescribed angles while leaving the genus, the other cones and $\im\rho$ alone.
Iterating produces quad meshes of the smallest size the corner count permits. We do not
prove that the iteration always terminates, so the classification cites its non-emptiness
input rather than reproving it.

Two extensions follow. On surfaces with boundary, which is the setting of parametrization
aligned to feature curves, the Reduction Lemma holds verbatim once the boundary turnings
are allowed into the relevant subgroup, and a feature component with an odd number of odd-angle corners
makes the holonomy irrelevant. At a fixed conformal structure the holonomy is not free but
computed, by an Abel--Jacobi condition. That is what explains why cones of odd topological
valence lie beyond methods built on quadratic differentials.
\end{abstract}

\tableofcontents

\section{Introduction}\label{sec:intro}

\subsection{The problem}

Let $M$ be a closed oriented surface of genus $g$ and let $C=\{c_1,\dots,c_n\}\subset M$ be
a finite set of marked points. A \emph{seamless parametrization} of $(M,C)$ is a locally
injective atlas on $M\setminus C$ with values in $\RR^2$ whose transition maps lie in
$\Zf\ltimes\RR^2$, that is, are of the form $z\mapsto i^k z+t$, and whose cone angles at the
points of $C$ are multiples of $\pi/2$. Such an atlas is what a quadrangulation or a
quad-dominant texture chart is built on. The transitions fix the two grid \emph{directions}
up to a quarter turn, so the cross field of axis directions pulls back consistently. The
integer \emph{lattice} does not, since an arbitrary translation $t$ moves it, and aligning
all charts to one $\ZZ^2$ is the separate quantization step of \S\ref{sec:open}.

The topological data of such an atlas is its \emph{holonomy signature}: the cone angles
$\theta_i=(m_i+4)\pi/2$, recorded by the integers $m_i>-4$, together with the rotational
holonomy
\[
  \rho\colon H_1(M\setminus C;\ZZ)\longrightarrow \Zf,
  \qquad \rho(\gamma_i)\equiv m_i \pmod 4,
\]
where $\gamma_i$ is a small loop around $c_i$. Equivalently, $\rho$ is the holonomy of the
cross field of the parametrization. Shen, Zhu, Capouellez, Panozzo, Campen and
Zorin~\cite{shen2022} asked:

\begin{quote}
  \emph{For which holonomy signatures do seamless parametrizations with the corresponding
  topology exist?}
\end{quote}

They proved a sufficient condition of gcd type. In their normalization a cone of angle
$(m+4)\pi/2$ has index $I=-m/4$, and the condition~\cite[Prop.~2]{shen2022} is
$\gcd(I_1,\dots,I_n)=\tfrac14$, that is
\[
  \gcd\nolimits_{\ZZ}(m_1,\dots,m_n)=1 .
\]
Under it every signature is realizable, in every genus, with a single exception: the torus
with exactly two cones, of angles $3\pi/2$ and $5\pi/2$. The condition holds as soon as one
cone has angle $3\pi/2$ or $5\pi/2$, since then some $|m_i|=1$. The authors note that it is
not necessary, and single out ``indices restricted to multiples of $\tfrac12$'', our even
case, as a realistic scenario it does not cover. The exception itself is a theorem of
Izmestiev, Kusner, Rote, Springborn and Sullivan~\cite{ikrss2013}, who proved the equivalent
combinatorial statement that the torus admits no $3,5$-quadrangulation. The signature
formulation and the constructions built on it have since been developed further, for
arbitrary cone configurations in~\cite{campen2019}, in Penner coordinates
in~\cite{penner2024}, and for surfaces with marked feature curves in~\cite{feature2025}, but
always on the sufficient side.

The gap has a definite shape. The known sufficient conditions are combinatorial, resting on
a rerouting argument, while the known obstruction is conformal and algebraic. This paper
closes the gap by working throughout in the language where the obstruction naturally lives.

\subsection{Results}

Write $D=\gen{m_1,\dots,m_n}\le\Zf$ for the subgroup generated by the cone data and $d$ for
the generator of $\im\rho$ in $\{1,2,4\}$, so that $\im\rho=d\Zf$. Since $\rho(\gamma_i)=m_i$,
every $m_i$ lies in $\im\rho$, so $D\subseteq d\Zf$; writing $D=d_D\Zf$ this says $d\mid d_D$.

\medskip\noindent\textbf{(a) The holonomy data collapses.} Theorem~\ref{thm:reduction}, the
Reduction Lemma, states that for $g\ge1$ two signatures with the same genus and the same
cone angles lie in the same $\MCG(M,C)$-orbit, after relabelling the marked points to match
the orders, if and only if they have the same $\im\rho$. Figure~\ref{fig:three} shows what
the three possibilities mean for the cross field. Since realizability is a mapping class
group invariant, the problem depends on at most three cases per angle multiset instead of
$4^{2g}$. Corollary~\ref{cor:odd} is the special case that makes contact with the
literature: if some $m_i$ is odd, in particular if some cone has angle $3\pi/2$ or $5\pi/2$,
then $D=\Zf$, a single orbit remains, and $\rho$ is irrelevant. This explains the shape of
the gcd condition of~\cite{shen2022} from first principles, and it also shows that condition
is more than is needed: coprimality of the $m_i$ is what their rerouting requires, one odd
$m_i$ is what realizability requires.

\medskip\noindent\textbf{(b) Realizability is non-emptiness of a stratum.}
Theorem~\ref{thm:main}: a signature is realizable if and only if the stratum of
\emph{primitive} $(4/d)$-differentials with orders $m_i/d$ on a genus-$g$ surface is
non-empty. The proof runs through the dictionary of \S\ref{sec:dictionary}, in which a
seamless parametrization is a flat cone metric with holonomy in $\Zf$, hence a meromorphic
section $q$ of $K^4$ with a zero or pole of order $m_i$ at $c_i$, and in which $\im\rho$
measures exactly how far $q$ is from being a power.

\medskip\noindent\textbf{(c) The classification.} Unwinding Theorem~\ref{thm:main} with the
known non-emptiness classifications for abelian and quadratic differentials gives
Theorem~\ref{thm:classification}: a signature fails to be realizable precisely in the five
families of Table~\ref{tab:exceptions}, in every genus. For $g\ge1$ the non-emptiness input
is a single cited theorem, (S7), which covers all three values of $k$ at once and whose
genus-one half is reproved here independently; genus $0$ is separate, and is
Proposition~\ref{prop:genus0}. Two of the five are in genus $1$ and are proved here from an
Abel--Jacobi criterion (Proposition~\ref{prop:genus1}); one of them is the exception
of~\cite{shen2022,ikrss2013}, recovered in one line. The other two live in genus $2$ and
appear to be new in this setting: they are the images under the dictionary of the empty
quadratic strata $Q(4)$ and $Q(1,3)$ of~\cite{masur1993}. In the language of cross fields, a
genus-two surface carrying a cross field whose holonomy along every homology loop is a
multiple of $\pi$ cannot be quadrangulated if its singularities are a single cone of angle
$6\pi$, or a pair of cones of angles $3\pi$ and $5\pi$.

\medskip\noindent\textbf{(d) An explicit construction.} \S\ref{sec:construction} makes the
non-emptiness half constructive. We exhibit an explicit square-tiled surface with one cone
in each genus (Lemma~\ref{lem:base}) and a local surgery (Lemma~\ref{lem:split}) that cuts a
saddle connection from a cone to itself and glues in one unit square, splitting a cone of
angle $w\pi/2$ into two of angles $(\gamma+2)\pi/2$ and $(w-\gamma+2)\pi/2$ while leaving
the genus, every other cone and $\im\rho$ untouched. Iterating with backtracking
(Algorithm~\ref{alg:construction}) produces a quad mesh with $2g-2+n$ squares, the minimum
the corner count allows, and each output is a witness whose genus, cone orders and $\im\rho$
can be read back off the gluing alone (Theorem~\ref{thm:construction}). We do not prove that
the iteration always terminates, so this is not by itself a theorem for all genera.

\medskip\noindent\textbf{(e) Feature curves.} \S\ref{sec:boundary} redoes the theory for
compact surfaces with boundary whose boundary must develop to axis-parallel segments, the
setting a feature network cuts a surface into. Gauss--Bonnet gains a corner term and forces
a parity constraint (Lemma~\ref{lem:gbboundary}); the space of signatures is again
$\Zf^{2g}$ (Lemma~\ref{lem:affineboundary}); and the Reduction Lemma holds verbatim with $D$
enlarged by the boundary turnings (Theorem~\ref{thm:reductionboundary}). Consequently a
boundary component of odd turning, equivalently one carrying an odd number of corners of odd
angle, forces $D=\Zf$, so the holonomy along homology loops cannot obstruct anything
(Corollary~\ref{cor:features}).

\medskip\noindent\textbf{(f) Relation to the fixed-structure criterion.} \S\ref{sec:fixed}
shows that at a fixed Riemann surface with fixed cone positions the holonomy is not free at
all but is \emph{computed} by an Abel--Jacobi condition (Proposition~\ref{prop:fixed}), so
the criterion of that line of work~\cite{qmg2,qmg3} is the pointwise condition and
Theorem~\ref{thm:main} asks whether the locus it cuts out is non-empty. The same statement
explains why cones of odd topological valence are inaccessible to methods built on quadratic
differentials (Corollary~\ref{cor:oddvalence}).

\begin{figure}[t]
\centering
\begin{tikzpicture}[scale=1.0,>=Stealth,font=\footnotesize,
  box/.style={rounded corners=2pt,fill=black!3,draw=black!20}]

\begin{scope}
  \draw[box] (-1.55,-0.15) rectangle (1.55,2.05);
  \node at (0,1.78) {$d=4$,\ \ $\im\rho=0$};
  \foreach \i in {0,1,2}\foreach \j in {0,1}{
    \draw[->,thick] ({-0.95+0.95*\i},{0.35+0.75*\j}) -- ++(0.5,0);}
  \node[align=center,font=\scriptsize] at (0,-0.62)
    {a global \emph{vector} field\\abelian differential};
\end{scope}

\begin{scope}[xshift=4.6cm]
  \draw[box] (-1.55,-0.15) rectangle (1.55,2.05);
  \node at (0,1.78) {$d=2$,\ \ $\im\rho=2\Zf$};
  \foreach \i in {0,1,2}\foreach \j in {0,1}{
    \begin{scope}[shift={({-0.95+0.95*\i},{0.35+0.75*\j})}]
      \draw[very thick] (-0.25,0) -- (0.25,0);
      \draw[thick,dashed] (0,-0.25) -- (0,0.25);
    \end{scope}}
  \node[align=center,font=\scriptsize] at (0,-0.62)
    {\emph{two} line fields, globally\\primitive quadratic differential};
\end{scope}

\begin{scope}[xshift=9.2cm]
  \draw[box] (-1.55,-0.15) rectangle (1.55,2.05);
  \node at (0,1.78) {$d=1$,\ \ $\im\rho=\Zf$};
  \foreach \i/\a in {0/0,1/25,2/55}\foreach \j/\b in {0/0,1/40}{
    \begin{scope}[shift={({-0.95+0.95*\i},{0.35+0.75*\j})},rotate=\a+\b]
      \draw[very thick] (-0.25,0) -- (0.25,0);
      \draw[very thick] (0,-0.25) -- (0,0.25);
    \end{scope}}
  \node[align=center,font=\scriptsize] at (0,-0.62)
    {the four prongs are interchangeable\\primitive $4$-differential};
\end{scope}
\end{tikzpicture}
\caption{The three possible values of $d$, the generator of $\im\rho$, and what each
means for the cross field. When $\im\rho=0$ the cross lifts to a global vector field. When
$\im\rho=2\Zf$ it does not, but it still splits globally into two distinguishable line
fields, drawn solid and dashed. When $\im\rho=\Zf$ transport around some loop cyclically
permutes the four prongs and no such splitting exists. By the Reduction Lemma these three
cases are all that matter: for fixed cone angles, $\im\rho$ alone decides the mapping class
group orbit, so realizability is a question about at most three signatures per angle
multiset rather than $4^{2g}$. Families 4 and 5 of Table~\ref{tab:exceptions} live in the
middle column.}
\label{fig:three}
\end{figure}
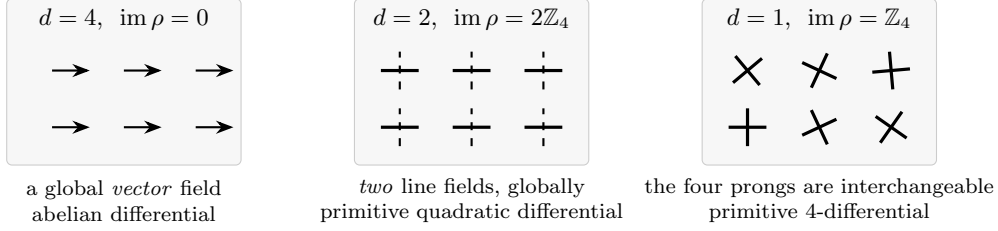

\subsection{What this adds to the sufficient condition}\label{sec:relation}

Since~\cite{shen2022} is both the source of the question and the only prior general answer,
it is worth saying precisely what changes. Their Proposition~2 is a one-way statement: if
$\gcd_{\ZZ}(m_1,\dots,m_n)=1$ then the signature is realizable, with the single torus
exception. Outside that hypothesis it says nothing, and as the paper itself notes the
condition ``is not necessary''. Theorem~\ref{thm:classification} fills the empty cell:

\begin{center}
\begin{tabular}{@{}lll@{}}
\toprule
 & $\gcd_{\ZZ}(m_i)=1$ & $\gcd_{\ZZ}(m_i)\ne1$ \\
\midrule
\cite[Prop.~2]{shen2022} & realizable; $1$ exception & no statement \\
here & realizable; the same $1$ exception & realizable; $4$ exceptions \\
\bottomrule
\end{tabular}
\end{center}

\noindent
Three things are behind that, and one thing is not improved.

\medskip\noindent\emph{(i) Sufficient becomes necessary and sufficient.} A combinatorial
construction can certify that a signature is realizable; it cannot certify that one is not.
Theorem~\ref{thm:main} converts the question into non-emptiness of a stratum of primitive
$k$-differentials, where the empty strata have been classified, and a complete list follows:
exactly five families fail, in every genus (Table~\ref{tab:exceptions}). One of the five is
the torus exception of~\cite{shen2022,ikrss2013}; the other four are the content of the
right-hand column above. Restricted to $\gcd_{\ZZ}(m_i)=1$ the classification returns their
statement, so nothing is contradicted, though not independently of them, since our positive
half cites (S7) where theirs is constructive.

\medskip\noindent\emph{(ii) Where the two overlap, less is needed.} Their rerouting moves a
loop's turning number by $I_i=-m_i/4$ for each cone it is rerouted around, so the reachable
adjustments are the integer combinations of the $I_i$, and hitting an arbitrary target in
$\tfrac14\ZZ$ forces $\gcd_{\ZZ}(m_i)=1$. It has to be an arbitrary target because the
construction of~\cite{campen2019} that they build on emits one fixed holonomy pattern along
the cut graph, which the signature must be made to match exactly.

The corresponding hypothesis here is weaker, and the reason is the difference between a
total and a residue. A turning number in $\tfrac14\ZZ$ is a count of quarter turns, in which
one full turn and three full turns are different numbers; realizability, by
Lemma~\ref{lem:d2}, depends only on the rotation class in $\Zf$, where they are the same.
Point pushing changes that class by $m_i$, so Corollary~\ref{cor:odd} needs only
$\gen{m_1,\dots,m_n}=\Zf$, one odd $m_i$, with no coprimality. The gap is not hypothetical:
at $g=4$ the cone angles $7\pi/2$ and $25\pi/2$ give $m=(3,21)$, so $\gcd_{\ZZ}=3$ and
Proposition~2 does not apply, while $3$ is odd, so a single mapping class group orbit
remains, $\rho$ is irrelevant, and the signature is realizable on eight squares, the minimum.

\medskip\noindent\emph{(iii) The region left open is decided.} The condition
$\gcd_{\ZZ}(m_i)\ne1$ is larger than it looks. It contains every signature whose cone angles
are all multiples of $\pi$, the case~\cite{shen2022} singles out as realistic and does not
cover; every signature with a single cone, since there $\gcd_{\ZZ}(m)=|m_1|=4(2g-2)$, which
is never $1$; and, by (ii), signatures with an odd cone angle as well.
Corollary~\ref{cor:evenregime} settles all of it at once: realizable except for four
signatures, all in genus $1$ or $2$. \S\ref{sec:examples} works through the cases. The one
to keep in mind is a genus-two surface with a single cone of angle $6\pi$, where the answer
is yes for $\im\rho=\Zf$, no for $2\Zf$, and yes for $0$. Same surface, same singularity,
and the holonomy alone decides. That is also why the reduction from $4^{2g}$ signatures to
three, item~(a) above, is doing real work rather than tidying.

\medskip\noindent\emph{What is not improved.} \cite{shen2022} is an algorithm that runs on
meshes, producing parametrizations whose distortion is then optimized; the present paper
decides existence and, in \S\ref{sec:construction}, builds minimal witnesses on a finite
range without a termination proof. On their own region their result is constructive and
self-contained where ours cites (S7). The two are complementary in the ordinary way: they
answer how to build one, we answer whether there is one to build.

\section{Seamless parametrizations and \texorpdfstring{$4$}{4}-differentials}\label{sec:setup}

\subsection{Conventions}

$M$ is a closed oriented surface of genus $g$, and $C=\{c_1,\dots,c_n\}\subset M$ a set of
distinct marked points, \emph{labelled}: the orders $m=(m_1,\dots,m_n)$ are an ordered
vector attached to $c_1,\dots,c_n$. Accordingly $\MCG(M,C)$ denotes the group of isotopy
classes of orientation-preserving homeomorphisms of $M$ fixing $C$ \emph{pointwise}; it does
not permute the marked points. Relabelling by a permutation $\pi$ of equal-order points is a
separate operation, used only where stated, and it identifies the signature $(m,\rho)$ with
$(m\circ\pi,\rho\circ\pi)$. Statements phrased in terms of the multiset of cone angles are
to be read modulo that identification. We write $\Zf=\ZZ/4\ZZ$ additively and identify it
with the group $\mu_4$ of quarter turns, $1\in\Zf$ being the rotation by $\pi/2$. Subgroups
of $\Zf$ are written $d\Zf$ with $d\in\{1,2,4\}$, so $4\Zf=0$.

\begin{definition}\label{def:signature}
A \emph{holonomy signature} on $(M,C)$ is a pair $s=(m,\rho)$ where
\begin{itemize}[itemsep=1pt]
  \item $m=(m_1,\dots,m_n)$ are integers $m_i>-4$, the \emph{orders}, encoding cone
  angles $\theta_i=(m_i+4)\pi/2$ and subject to Gauss--Bonnet
  $\sum_i m_i=4(2g-2)$;
  \item $\rho\colon H_1(M\setminus C;\ZZ)\to\Zf$ is a homomorphism with
  $\rho(\gamma_i)=m_i\bmod 4$.
\end{itemize}
An order $m_i=0$ is a \emph{marked point}: the metric is smooth there.
\end{definition}

Since $H_1(M\setminus C;\ZZ)\cong\ZZ^{2g}\oplus\big(\bigoplus_i\ZZ\gamma_i\big)/(\sum_i\gamma_i)$
and Gauss--Bonnet gives $\sum_i m_i\equiv0\pmod4$, such a $\rho$ exists, and the set
\[
  A(m)=\{\rho:\rho(\gamma_i)=m_i\ \forall i\}
\]
is a coset of $H^1(M;\Zf)$ inside $H^1(M\setminus C;\Zf)$, the inclusion
$H^1(M;\Zf)\hookrightarrow H^1(M\setminus C;\Zf)$ being injective. Signatures are thus
\emph{characters}, and the space they form is cohomological. Fixing a symplectic basis
$a_1,b_1,\dots,a_g,b_g$ of $H_1(M;\ZZ)$ identifies $H^1(M;\Zf)$ with $\Zf^{2g}$, and
Poincar\'e duality on the closed oriented surface $M$ identifies it with $H_1(M;\Zf)$,
carrying the cup product to the intersection form. We use that identification silently
below: the symplectic group acts on $H^1(M;\Zf)\cong\Zf^{2g}$, while curves, meaning the
loops $\gamma_i$, the $a_j,b_j$, and the arcs along which cones are pushed, live in $H_1$.

A seamless parametrization is \emph{realized on a mesh} if it is piecewise linear with
respect to some refinement of a given triangulation of $M$ whose vertex set contains $C$. A
signature is \emph{realizable} if some seamless parametrization has it. We also use the
following invariants of a signature:
\[
  D=\gen{m_1\bmod4,\dots,m_n\bmod4}\le\Zf,
  \qquad
  \im\rho=d\Zf \ \ (d\in\{1,2,4\}),
  \qquad D\subseteq\im\rho .
\]

\subsection{Quad meshes}\label{sec:mesh}

\begin{remark}[square-tiled surfaces]\label{rem:mesh}
A closed quad mesh built from $N$ unit squares is a seamless parametrization: each square is
a chart, the transitions are rotations by multiples of $\pi/2$ and translations, and a
vertex of valence $v$ is a cone of angle $v\pi/2$, that is of order $m=v-4$. Every corner
belongs to exactly one vertex, so $4N=\sum_i v_i$. Such a mesh is thus an existence
certificate for its own signature, and this is the form in which all constructions in this
paper are given.

We fix the combinatorial model once. The \emph{darts} are the pairs (face, side), encoded as
$4f+s$ with $s\in\{0,1,2,3\}$: dart $4f+s$ is side $s$ of square $f$, traversed
counterclockwise. Then $\sigma(4f+s)=4f+(s+1\bmod4)$ is the face rotation, the gluing is a
fixed-point-free involution $\alpha$ on darts, and the vertices are the orbits of
$\nu=\alpha\circ\sigma^{-1}$, the orbit through a dart consisting of the darts whose tail is
that vertex, so that its length is the valence. Orienting side $s$ so that its outward
normal points in direction $s$, gluing side $s$ of $f$ to side $t$ of $h$ makes the chart
transition from $f$ to $h$ the rotation
\[
  r=(t-s+2)\bmod4 ,
\]
which is antisymmetric under swapping the two darts, as it must be. In particular a gluing
of two \emph{opposite} sides, $t=s+2$, is a translation, $r=0$, while a gluing of a side to
itself in reverse is a half-turn. Finally $\rho$ is read off the dual graph: the dual
$1$-skeleton is a deformation retract of $M$ minus the vertices, so its cycles generate
$H_1(M\setminus C)$, and $\rho$ of a cycle is the sum of the rotations it crosses.
\end{remark}

\subsection{Results used as black boxes}\label{sec:blackbox}

The following are used without proof. For $g=0$ the classification uses Troyanov (S6),
through Proposition~\ref{prop:genus0}; for $g\ge1$ it uses the Gendron--Tahar classification
of primitive $k$-differential strata, (S7). Beyond routine topology those are the only two
that matter.

\begin{description}[leftmargin=2.2em,itemsep=1pt]
  \item[(S1)] A flat cone metric on a closed surface admits a geodesic triangulation
  whose vertex set contains the cone points.
  \item[(S2)] Two triangulations of a compact surface admit a common refinement.
  \item[(S3)] $\Sp(2g,\ZZ)$ acts transitively on primitive vectors of $\ZZ^{2g}$.
  \item[(S4)] $\Sp(2g,\ZZ)\to\Sp(2g,\ZZ/N)$ is surjective.
  \item[(S5)] $\MCG(\Sigma_{g,1},\partial)\to\Sp(2g,\ZZ)$ is surjective.
  \item[(S6)] (Troyanov~\cite{troyanov1991}) On a closed surface, for $n\ge3$ distinct
  points and angles $\theta_i>0$ satisfying Gauss--Bonnet there is a flat cone metric
  with those angles at those points.
  \item[(S7)] (Gendron--Tahar~\cite[Th.~1.4]{gt2020}) Let $g\ge1$, let $k\ge1$, and let
  $\mu=(a_1,\dots,a_n)$ be a partition of $k(2g-2)$ with every $a_i>-k$. The stratum
  $\Omega^k\mathcal M_g(\mu)$ of \emph{primitive} $k$-differentials of profile $\mu$,
  those that are not powers of $k'$-differentials for $1\le k'<k$, is empty if and only
  if
  \begin{enumerate}[label=\emph{(\roman*)},itemsep=0pt,topsep=2pt]
    \item $g=1$ and $\mu=(1,-1)$;
    \item $g=1$ and $\mu=\emptyset$ and $k\ge2$;
    \item $g=2$, $k=2$ and $\mu=(4)$ or $\mu=(3,1)$.
  \end{enumerate}
\end{description}

The hypothesis $g\ge1$ is the scope of the section of~\cite{gt2020} in which Theorem~1.4
appears, and it cannot be dropped. For $g=0$, $k=4$ and $\mu=(-2,-2,-2,-2)$, which is
admissible since $\sum_ia_i=-8=k(2g-2)$ and $a_i>-4$, the only $4$-differential with that
divisor on $\PP^1$ is $c\,(dz)^4/\prod_i(z-p_i)^2$, the square of
$\sqrt c\,(dz)^2/\prod_i(z-p_i)$, so that stratum is empty although it matches no case
above. Genus $0$ is handled separately, by Proposition~\ref{prop:genus0}.

For $g\ge1$ this one statement supplies every non-emptiness input the classification needs,
for all three values of $k=4/d$ at once. Its case (iii) for $k=2$ is the classical theorem of
Masur and Smillie~\cite{masur1993}, and its cases (i) and (ii) are reproved independently
here as Proposition~\ref{prop:genus1}. What we take purely on trust is that the list stops
there.

\subsection{The dictionary}\label{sec:dictionary}

This subsection proves the three correspondences summarized in
Figure~\ref{fig:dictionary}.

\begin{figure}[t]
\centering
\begin{tikzpicture}[font=\footnotesize,>=Stealth,
  bx/.style={draw,thick,rounded corners=3pt,fill=black!3,align=center,
             minimum width=3.5cm,minimum height=1.15cm,inner sep=4pt}]

\node[bx] (A) at (0,0) {\textbf{seamless parametrization}\\ atlas on $M\setminus C$,\\
  transitions $z\mapsto i^kz+t$};
\node[bx] (B) at (5.3,0) {\textbf{flat $\Zf$ cone metric}\\ cone angles in $(\pi/2)\ZZ$,\\
  holonomy in $\Zf$};
\node[bx] (C) at (10.6,0) {\textbf{meromorphic $4$-differential}\\ $(X,q)$ on a Riemann\\
  surface $X$};

\draw[<->,thick] (A) -- node[above,font=\scriptsize] {Lem.~\ref{lem:d1}} (B);
\draw[<->,thick] (B) -- node[above,font=\scriptsize] {Lem.~\ref{lem:d3}} (C);

\node[font=\scriptsize,align=center] at (0,-1.35)
  {cone angle $\theta_i=(m_i+4)\frac\pi2$\\ rotational holonomy $\rho$};
\node[font=\scriptsize,align=center] at (5.3,-1.35)
  {$|q|^{1/2}$ is the metric\\ $\rho$ is its holonomy};
\node[font=\scriptsize,align=center] at (10.6,-1.35)
  {$\divisor(q)=\sum_i m_ic_i$\\ $\im\rho\subseteq d\Zf\iff q=\eta^{d}$};

\draw[decorate,decoration={brace,amplitude=4pt,mirror}] (-1.85,-2.1) -- (12.45,-2.1);
\node[align=center] at (5.3,-2.75)
  {\textbf{Theorem~\ref{thm:main}:}\ \ $s$ is realizable
   $\iff$ the stratum of primitive $(4/d)$-differentials\\
   with orders $m_i/d$ on a genus-$g$ surface is non-empty};
\end{tikzpicture}
\caption{The correspondences of \S\ref{sec:dictionary}. Reading left to right, a seamless
parametrization is a flat cone metric with holonomy in $\Zf$ (Lemma~\ref{lem:d1}), and such
a metric is $|q|^{1/2}$ for a meromorphic $4$-differential $q$ whose zeros and poles are the
cones (Lemma~\ref{lem:d3}). Under this correspondence $\im\rho$ measures primitivity: $q$ is
a $d$-th power exactly when $\im\rho\subseteq d\Zf$, so $\im\rho=d\Zf$ says that $q=\eta^{d}$
with $\eta$ \emph{primitive} (Lemma~\ref{lem:d4}). Combined with the Reduction Lemma this
turns the question into one about strata, which \S\ref{sec:classification} answers.}
\label{fig:dictionary}
\end{figure}
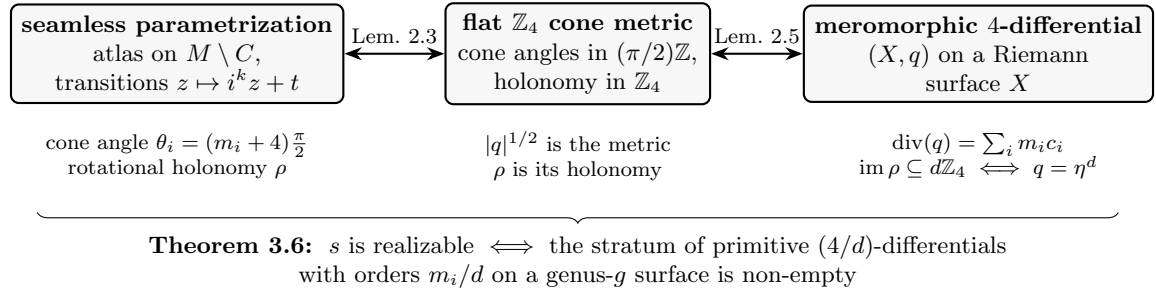

\begin{lemma}[Seamless parametrization $=$ flat $\Zf$ cone metric]\label{lem:d1}
Seamless parametrizations of $(M,C)$ are the same thing as flat cone metrics on $M$
with cone points contained in $C$, cone angles in $(\pi/2)\ZZ$, and rotational
holonomy contained in $\Zf$. The signature corresponds on both sides.
\end{lemma}

\begin{proof}
Given a seamless atlas, pull back the euclidean metric; this is well defined since the
transitions $z\mapsto i^kz+t$ are isometries. Parallel transport along a loop $\gamma$ is
the rotation part $i^{k(\gamma)}$ of the composite transition, and rotation parts multiply
in the abelian group $\mu_4$, so $\gamma\mapsto k(\gamma)$ is the rotational holonomy
$H_1(M\setminus C;\ZZ)\to\Zf$.

Conversely, developing maps on simply connected subsets of $M\setminus C$ give an atlas
whose transitions are orientation-preserving euclidean isometries and are local isometries,
hence locally injective. The rotation part around any loop is the rotational holonomy, which
lies in $\Zf$ by hypothesis, so on a good cover the transitions may be taken of the form
$z\mapsto i^kz+t$. The holonomy around $\gamma_i$ is rotation by $\theta_i$, forcing
$\rho(\gamma_i)=m_i\bmod4$ as in Definition~\ref{def:signature}.
\end{proof}

\begin{lemma}[Realizability is topological]\label{lem:d2}
\emph{(a)} The set of realizable signatures is invariant under $\MCG(M,C)$.
\emph{(b)} A signature $s$ is realizable if and only if there is a flat cone metric
on $M$ with cone angles $(m_i+4)\pi/2$ at $n$ distinct points whose signature lies in
the $\MCG(M,C)$-orbit of $s$.
\end{lemma}

\begin{proof}
(a) If $\varphi\in\MCG(M,C)$ and $\{U_a,f_a\}$ realizes $s$, then
$\{\varphi^{-1}(U_a),f_a\circ\varphi\}$ is again a seamless atlas with the same transitions,
and its signature is $\varphi^*s$. Any mapping class is represented by a PL homeomorphism,
so the atlas is again PL on a refinement of a triangulation.

(b) Necessity is Lemma~\ref{lem:d1}. For sufficiency, let $\mu$ be a flat cone metric with
cone points $p_1,\dots,p_n$ and signature in the orbit of $s$. A homeomorphism $\psi$ with
$\psi(c_i)=p_{\sigma(i)}$ matching orders makes $\psi^*\mu$ have cone set exactly $C$, and
by (a) we may assume its signature is $s$; Lemma~\ref{lem:d1} then gives a seamless atlas.
To make it PL on a refinement of the given triangulation $T$: by (S1) the metric has a
geodesic triangulation $T''$ containing the cone points among its vertices, and the
developing charts are affine on each of its triangles; by (S2) a common refinement of $T$
and $T''$ works.
\end{proof}

Two consequences are worth isolating, because they are what makes the rest possible: the
conformal structure is free, and so are the cone positions. This is the precise sense in
which the question differs from the fixed-conformal-structure criterion discussed in
\S\ref{sec:fixed}.

\begin{lemma}[Flat $\Zf$ cone metric $=$ meromorphic $4$-differential]\label{lem:d3}
Flat cone metrics on $M$ with cone angles in $(\pi/2)\ZZ$ and rotational holonomy in
$\Zf$ are the same thing as pairs $(X,q)$ with $X$ a Riemann surface homeomorphic to
$M$ and $q$ a meromorphic section of $K_X^{4}$ with
$\divisor(q)=\sum_i m_ic_i$, $m_i>-4$. The metric is $|q|^{1/2}$, the cone angle at
$c_i$ is $(m_i+4)\pi/2$, and the rotational holonomy is the monodromy of the local
fourth roots of $q$.
\end{lemma}

\begin{proof}
Let $\mu$ be such a metric, with developing charts $z_a$ and transitions $z_b=i^kz_a+t$.
Then $(dz_b)^4=i^{4k}(dz_a)^4=(dz_a)^4$, so the local $4$-differentials $(dz_a)^4$ agree on
overlaps and define a nowhere-zero holomorphic section $q$ of $K^4$ on $M\setminus C$ with
$|q|^{1/2}=\mu$. Near a cone of angle $(m+4)\pi/2$ the metric is isometric to the standard
cone, developed in a holomorphic coordinate $w$ by $z=\tfrac{4}{m+4}w^{(m+4)/4}$; hence
\[
  q=(dz)^4=w^{m}(dw)^4 ,
\]
so $q$ extends meromorphically across $c_i$ with order exactly $m_i$, and $m_i>-4$ is
precisely positivity of the cone angle. Since $\deg K_X^4=4(2g-2)$, the degree of
$\divisor(q)$ reproduces Gauss--Bonnet. Conversely, given $(X,q)$, the metric $|q|^{1/2}$ is
flat away from $\divisor(q)$ with the stated cone angles; a local primitive
$z=\int q^{1/4}$ is a developing chart, two choices differ by $z\mapsto i^kz+t$, and so the
rotational holonomy is the monodromy of $q^{1/4}$ and lies in $\mu_4=\Zf$.
\end{proof}

Call a $k$-differential \emph{primitive} if it is not an $e$-th power for any $e>1$ dividing
$k$; every abelian differential is primitive.

\begin{lemma}[Primitivity is measured by $\im\rho$]\label{lem:d4}
Let $q$ be as in Lemma~\ref{lem:d3}, with rotational holonomy $\rho$, and let
$e\mid 4$. Then $q=\eta^{e}$ for a meromorphic section $\eta$ of $K^{4/e}$ if and only
if $\im\rho\subseteq e\Zf$. Consequently, if $d$ generates $\im\rho$ then $q=\eta^{d}$
for a primitive $(4/d)$-differential $\eta$ with $\divisor(\eta)=\sum_i (m_i/d)c_i$;
and conversely $\eta^{d}$ has rotational holonomy with image exactly $d\Zf$.
\end{lemma}

\begin{proof}
Work locally, $q=f\,(dz)^4$. A fourth root of $q$ is $f^{1/4}dz$, and by Lemma~\ref{lem:d3}
the holonomy $\rho$ is the monodromy character of the multivalued $f^{1/4}$, valued in
$\mu_4$. An $e$-th root of $q$ is $h(dz)^{4/e}$ with $h^e=f$, that is
$h=f^{1/e}=(f^{1/4})^{4/e}$, which is single valued if and only if $\rho(\gamma)^{4/e}=1$
for all $\gamma$; additively this reads $(4/e)\rho(\gamma)=0$, that is
$\rho(\gamma)\in e\Zf$. This also forces $e\mid m_i$, since $\rho(\gamma_i)=m_i\bmod 4$ and
$e\mid4$.

Now let $d$ generate $\im\rho$. By the above there is $\eta$ with $\eta^d=q$ and
$\divisor(\eta)=\divisor(q)/d$. Its own holonomy $\rho_\eta\colon H_1\to\ZZ_{4/d}$ satisfies
$\rho=d\cdot\rho_\eta$ under $\ZZ_{4/d}\cong d\Zf$, so $\im\rho_\eta$ is all of
$\ZZ_{4/d}$; by the first claim applied to $\eta$ this says exactly that $\eta$ has no
$e$-th root for $e>1$ dividing $4/d$. The converse is the same computation read backwards.
\end{proof}

\section{The classification}\label{sec:classification}

\subsection{The mapping class group action}\label{sec:reduction}

Throughout this subsection $g\ge1$ and the orders $m$ are fixed. Recall that $A(m)$ is a
coset of $H^1(M;\Zf)\cong\Zf^{2g}$ and that $D=\gen{m_1,\dots,m_n}$.

\begin{figure}[t]
\centering
\begin{tikzpicture}[font=\footnotesize,>=Stealth,scale=1.0,
  cone/.style={circle,fill=black,inner sep=1.7pt},
  ghost/.style={circle,draw,thick,fill=white,inner sep=1.5pt}]

\newcommand{\srf}{%
  \draw[thick,fill=black!4,rounded corners=16pt] (-1.9,-0.85) rectangle (1.9,0.85);
  \draw[thick,fill=white] (-0.85,0) ellipse (0.36 and 0.22);
  \draw[thick,fill=white] ( 0.85,0) ellipse (0.36 and 0.22);}

\begin{scope}[scale=1.15]
  \srf
  \draw[dashed,thick,gray] (-0.85,0) ellipse (0.62 and 0.48);
  \node[gray] at (-1.62,0.52) {$\alpha$};
  \node[cone,label={[inner sep=1.5pt]below:$p$}] at (-0.85,-0.48) {};
  \draw[->,very thick] (-0.32,-0.30) arc (-38:198:0.58 and 0.46);
  \node[align=center,font=\scriptsize] at (0,-1.55)
    {push the cone $p$ once around $\alpha$};
\end{scope}

\begin{scope}[xshift=7.2cm]
  \node[align=center] at (0,0.45)
    {$P(\alpha)^{*}\rho \;=\; \rho \;+\; m_p\,\langle\,\cdot\,,\alpha\rangle$};
  \node[align=center,font=\scriptsize] at (0,-0.45)
    {$\rho$ is unchanged on every $\gamma_i$,\\
     and shifts by $m_p$ along each loop meeting $\alpha$ once};
  \node[align=center,font=\scriptsize] at (0,-1.55)
    {so pushes generate translations by $D\cdot H^1(M;\Zf)$};
\end{scope}
\end{tikzpicture}
\caption{Point pushing (Lemma~\ref{lem:push}). Dragging a cone of order $m_p$ once around a
simple closed curve $\alpha$ is a mapping class fixing $C$ pointwise, and it changes the
holonomy by $m_p$ times the intersection pairing with $\alpha$. The loops $\gamma_i$ around
the cones are left alone, so the cone angles do not move. Ranging over $p$ and $\alpha$,
these pushes translate $A(m)$ by the subgroup $D\cdot H^1(M;\Zf)$, where
$D=\gen{m_1,\dots,m_n}$. When some $m_p$ is odd we have $D=\Zf$ and the pushes alone act
transitively, which is Corollary~\ref{cor:odd}.}
\label{fig:push}
\end{figure}

\begin{lemma}[Point pushing]\label{lem:push}
Let $p\in C$ and let $\alpha\in H_1(M;\ZZ)$ be represented by a simple closed curve
through $p$, and let $P(\alpha)\in\MCG(M,C)$ be the corresponding point-pushing map.
Then
\[
  P(\alpha)_*x=x+\langle x,\alpha\rangle\gamma_p
  \quad\text{for } x\in H_1(M\setminus C;\ZZ),
  \qquad\text{hence}\qquad
  P(\alpha)^*\rho=\rho+m_p\langle\cdot,\alpha\rangle .
\]
As $p$ and $\alpha$ vary these moves generate exactly the translations of $A(m)$ by
$D\cdot H^1(M;\Zf)$, and they fix every puncture value. See Figure~\ref{fig:push}.
\end{lemma}

\begin{proof}
$P(\alpha)$ is the composition $T_{\alpha_L}T_{\alpha_R}^{-1}$ of Dehn twists along the two
boundary curves of an annulus neighbourhood of the curve, with $p$ in the annulus between
them. On $H_1(M\setminus C)$ these are transvections, and since the two curves are disjoint,
$\langle\alpha_L,\alpha_R\rangle=0$ and the twists commute, so
\[
  P(\alpha)_*x
  = x+\langle x,\alpha_L\rangle\alpha_L-\langle x,\alpha_R\rangle\alpha_R .
\]
Both boundary curves are disjoint from $C$ and homologous to $\alpha$ in $M$, so their
intersection numbers with $x$ agree and equal $\langle x,\alpha\rangle$; and their
difference in $H_1(M\setminus C)$ is $\gamma_p$, because the annulus between them contains
exactly the puncture $p$. Hence $P(\alpha)_*x=x+\langle x,\alpha\rangle\gamma_p$ and
$(P(\alpha)^*\rho)(x)=\rho(x)+m_p\langle x,\alpha\rangle$. Puncture values are unchanged
because $\langle\gamma_i,\alpha\rangle=0$. Finally
$\alpha\mapsto\langle\cdot,\alpha\rangle$ is Poincar\'e duality, so already for $\alpha$ in
a symplectic basis these functionals span $m_pH^1(M;\Zf)$; summing over $p$ gives
$D\cdot H^1(M;\Zf)$, and no more, since every such move is of this form.
\end{proof}

\begin{lemma}[Symplectic transitivity]\label{lem:sp}
Let $g\ge1$, $N\ge1$. For $v\in(\ZZ/N)^{2g}$ set
$\content(v)=\gcd(N,v_1,\dots,v_{2g})$. Then $\Sp(2g,\ZZ/N)$ preserves $\content$
and acts transitively on each of its level sets.
\end{lemma}

\begin{proof}
Invariance is clear, as $S$ is invertible over $\ZZ/N$. For transitivity, let
$\content(v)=c$, write $N=cN'$ and $v=ca$ with $a\in(\ZZ/N')^{2g}$ and
$\gcd(N',a_1,\dots,a_{2g})=1$.

\emph{Step 1: lift $a$ to a primitive integer vector.} Adjusting one coordinate does not
suffice: for $N'=5$ and $a=(2,0,\dots,0)$ every lift $(2+5k,0,\dots,0)$ has content
$|2+5k|\ne1$. So we move two, which is possible as $2g\ge2$.

Choose a lift $\tilde a$ with $\tilde a_1\ne0$ (add a multiple of $N'$ if needed) and set
$D=\gcd(\tilde a_1,\tilde a_3,\dots,\tilde a_{2g})>0$, which omits $\tilde a_2$ and is
therefore fixed once and for all. Let $p\mid D$ be prime. If $p\mid N'$ then
$p\nmid\tilde a_2$, since otherwise $p$ would divide $N'$ and every $\tilde a_i$, against
$\gcd(N',a)=1$; so $\tilde a_2+N't$ is prime to $p$ for every $t$. If $p\nmid N'$, then
$N'$ is invertible mod $p$ and $t\equiv(1-\tilde a_2)(N')^{-1}$ makes
$\tilde a_2+N't\equiv1$. Choosing such a $t$ simultaneously for the finitely many primes of
the second kind, by the Chinese Remainder Theorem, gives $\gcd(D,\tilde a_2+N't)=1$, hence
$\gcd(\tilde a_1,\tilde a_2+N't,\tilde a_3,\dots,\tilde a_{2g})=1$: a primitive lift.

\emph{Step 2.} By (S3) there is $\tilde S\in\Sp(2g,\ZZ)$ with $\tilde S\tilde a=e_1$,
hence $\tilde S(c\tilde a)=ce_1$.

\emph{Step 3.} By (S4) the reduction $S$ of $\tilde S$ lies in $\Sp(2g,\ZZ/N)$ and
$Sv=ce_1$. Every $v$ of content $c$ maps to the same $ce_1$, so the level set is one orbit.
\end{proof}

\begin{lemma}[A base point with full symplectic stabilizer]\label{lem:basepoint}
There exist $\rho_0\in A(m)$ and a subgroup $H\le\MCG(M,C)$ fixing $C$ pointwise such
that $H$ fixes $\rho_0$ and the image of $H$ in $\operatorname{Aut}(H^1(M;\Zf))$ is all
of $\Sp(2g,\Zf)$.
\end{lemma}

\begin{proof}
Choose an embedded closed disk $D_0\subset M$ containing $C$ in its interior and let
$\Sigma=M\setminus\operatorname{int}(D_0)$, a genus-$g$ surface with one boundary circle.
Choose a symplectic basis $a_1,b_1,\dots,a_g,b_g$ of $H_1(M;\ZZ)$ represented by simple
closed curves in the interior of $\Sigma$. Define $\rho_0$ by $\rho_0(\gamma_i)=m_i$ and
$\rho_0(a_j)=\rho_0(b_j)=0$; this is well defined because the only relation among the
generators is $\sum_i\gamma_i=0$ and $\sum_i m_i\equiv0$. Note that $\rho_0$ vanishes on the
image of $H_1(\Sigma;\ZZ)\to H_1(M\setminus C;\ZZ)$, which is generated by the $a_j,b_j$ and
by $[\partial\Sigma]=\sum_i\gamma_i$.

Let $H$ be the image of $\MCG(\Sigma,\partial\Sigma)\to\MCG(M,C)$, extending by the identity
on $D_0$. Every $\varphi\in H$ is the identity near $C$, so $\varphi_*\gamma_i=\gamma_i$;
and for any $1$-cycle $x$ the class $\varphi_*x-x$ lies in the image of $H_1(\Sigma)$, since
$\varphi$ is the identity outside $\Sigma$ and near $\partial\Sigma$, so $x$ and
$\varphi_*x$ agree outside $\Sigma$ and their arcs inside have the same endpoints. Hence
$(\varphi^*\rho_0)(x)=\rho_0(x)+\rho_0(\varphi_*x-x)=\rho_0(x)$. Finally
$\MCG(\Sigma,\partial\Sigma)\to\Sp(2g,\ZZ)$ is surjective by (S5) and
$\Sp(2g,\ZZ)\to\Sp(2g,\Zf)$ is surjective by (S4).
\end{proof}

With $\rho_0$ fixed, identify $A(m)$ with $\Zf^{2g}$ by $\rho=\rho_0+v$. Then $H$ acts
linearly, $v\mapsto Sv$, and point pushing acts by translations $v\mapsto v+w$,
$w\in D^{2g}$, and
\begin{equation}\label{eq:image}
  \im(\rho_0+v)=\gen{D,v_1,\dots,v_{2g}}=\gcd(d_D,v_1,\dots,v_{2g})\,\Zf ,
\end{equation}
where $d_D$ generates $D$.

\begin{theorem}[Reduction Lemma]\label{thm:reduction}
Let $g\ge1$ and fix the orders $m$. Two signatures with orders $m$ lie in the same
$\MCG(M,C)$-orbit if and only if they have the same $\im\rho$. Hence the orbits with
orders $m$ are in bijection with the subgroups of $\Zf$ containing $D$: there are
$1$, $2$ or $3$ of them according as $D=\Zf$, $D=2\Zf$, $D=0$.
\end{theorem}

\begin{proof}
If $\rho'=\rho\circ\varphi_*$ then $\im\rho'=\rho(\varphi_*H_1)=\rho(H_1)=\im\rho$, so the
invariant is constant on orbits.

Conversely let $v,w\in\Zf^{2g}$ with $\gcd(d_D,v)=\gcd(d_D,w)$, by~\eqref{eq:image}. Reduce
modulo $D$: since $d_D\mid4$ we have $\Zf/D=\ZZ/d_D$, and the point-pushing translations by
$D^{2g}$ act transitively on each fibre of $\Zf^{2g}\to(\ZZ/d_D)^{2g}$. The reductions
$\bar v,\bar w$ have equal content in $(\ZZ/d_D)^{2g}$, so by Lemma~\ref{lem:sp} with
$N=d_D$ there is $\bar S\in\Sp(2g,\ZZ/d_D)$ with $\bar S\bar v=\bar w$; by (S4) it lifts to
$S\in\Sp(2g,\Zf)$, which by Lemma~\ref{lem:basepoint} is realized by some $\varphi\in H$.
Applying $\varphi$ and then a point-pushing translation carries $v$ to $w$.
\end{proof}

\begin{corollary}\label{cor:odd}
If some cone angle is an odd multiple of $\pi/2$, in particular if some cone has angle
$3\pi/2$ or $5\pi/2$, then $D=\Zf$, all signatures with those cone angles form a single
$\MCG$-orbit, and realizability does not depend on $\rho$.
\end{corollary}

\begin{proof}
$m_i$ odd gives $D=\Zf$ and $d_D=1$, so $(\ZZ/d_D)^{2g}$ is trivial and
Theorem~\ref{thm:reduction} leaves one orbit.
\end{proof}

Mechanically, pushing a cone of odd order once around a handle shifts the holonomy along the
dual loop by an odd amount, so repeated pushes reach every value. This is where the gcd
hypothesis of~\cite{shen2022} comes from, and it is also the point at which less turns out
to be needed: their rerouting must reproduce turning numbers in $\tfrac14\ZZ$ exactly and so
asks for $\gcd_{\ZZ}(m_i)=1$, while realizability sees only the class of $\rho$ in $\Zf$ and
so asks for $\gen{m_i}=\Zf$. See \S\ref{sec:relation}(ii).

\subsection{Realizability is non-emptiness of a stratum}\label{sec:main}

\begin{theorem}\label{thm:main}
Let $s=(m,\rho)$ be a holonomy signature on $(M,C)$, $g=\operatorname{genus}(M)$, and
let $d$ generate $\im\rho$. Then $s$ is realizable if and only if the stratum of
primitive $(4/d)$-differentials with orders $(m_1/d,\dots,m_n/d)$ on a genus-$g$
surface is non-empty.
\end{theorem}

\begin{proof}
If $s$ is realizable, Lemmas~\ref{lem:d1} and~\ref{lem:d3} give $(X,q)$ with
$\divisor(q)=\sum m_ic_i$ and holonomy $\rho$, and Lemma~\ref{lem:d4} extracts a primitive
$(4/d)$-differential $\eta$ with orders $m_i/d$.

Conversely let $\eta$ be a primitive $(4/d)$-differential with orders $m_i/d$ on a genus-$g$
Riemann surface $X$, with zeros and poles at $p_1,\dots,p_n$, and put $q=\eta^{d}$. By
Lemma~\ref{lem:d4} the holonomy of $q$ has image exactly $d\Zf$, and
$\divisor(q)=\sum_i m_ip_i$, so by Lemma~\ref{lem:d3} the metric $|q|^{1/2}$ is a flat cone
metric on $X$ with the prescribed cone angles. Transport it to $M$ by an
orientation-preserving homeomorphism $X\to M$ carrying $p_i$ to $c_{\sigma(i)}$ for a
permutation $\sigma$ matching equal orders; the resulting signature $s'$ has the same
multiset of orders as $s$ and $\im\rho'=d\Zf=\im\rho$. For $g\ge1$,
Theorem~\ref{thm:reduction} puts $s'$ in the orbit of $s$; for $g=0$ we have
$H^1(M;\Zf)=0$, so $A(m)$ is a single point and $s'=s$ after relabelling $C$. Either way
Lemma~\ref{lem:d2}(b) gives realizability of $s$.
\end{proof}

The problem is therefore \emph{equivalent} to a stratum non-emptiness question. For $d=4$
these are strata of abelian differentials, for $d=2$ strata of primitive quadratic
differentials, and for $d=1$ strata of primitive $4$-differentials.

This is the point at which the question meets a developed theory. Strata of abelian and
quadratic differentials, their non-emptiness and their connected components, are
classical~\cite{masur1993,kontsevichzorich,lanneau}; existence with prescribed singularities
has since been settled uniformly in $k$ by Gendron and Tahar~\cite{gt2020,gt2021,gt2025q},
whose Theorem~1.4 is quoted as (S7) and is the one external input the classification uses.
The construction of \S\ref{sec:construction} is independent of it, and over the genus-two
quadratic strata it reproduces that theorem's exceptional list exactly.

\subsection{Genus zero and genus one}\label{sec:lowgenus}

\begin{proposition}\label{prop:genus0}
Every Gauss--Bonnet-admissible holonomy signature on the sphere is realizable.
\end{proposition}

\begin{proof}
First, $n\ge3$: Gauss--Bonnet gives $\sum_i m_i=-8$ while $m_i>-4$ forces
$\sum_i m_i>-4n$. Second, $\rho$ is determined by $m$, because $H_1(S^2\setminus C;\ZZ)$ is
generated by the $\gamma_i$; in particular any flat cone metric on $S^2$ with cone angles in
$(\pi/2)\ZZ$ automatically has holonomy in $\Zf$, and its signature is the given one. Third,
such a metric exists by (S6). Now apply Lemma~\ref{lem:d2}(b).
\end{proof}

For genus $1$ the answer comes from Abel--Jacobi, in a form that treats all three values of
$d$ at once.

\begin{proposition}\label{prop:genus1}
Let $k\in\{1,2,4\}$ and let $\mu=(\mu_1,\dots,\mu_n)$ be integers with $\mu_i>-k$ and
$\sum_i\mu_i=0$. Let $n_0$ be the number of nonzero $\mu_i$. The stratum of primitive
$k$-differentials with orders $\mu$ on a genus-$1$ surface is non-empty if and only if
\begin{enumerate}[label=\emph{(\roman*)},itemsep=1pt]
  \item $n_0=0$ and $k=1$; or
  \item $n_0=2$, in which case the nonzero orders are $(\mu,-\mu)$, and $|\mu|\ge2$; or
  \item $n_0\ge3$.
\end{enumerate}
\emph{(}$n_0=1$ cannot occur, since the orders sum to zero.\emph{)}
\end{proposition}

\begin{proof}
On $E=\CC/\Lambda$ the canonical bundle is trivialized by $dz$, so every meromorphic
$k$-differential is $q=f(dz)^k$ with $f$ meromorphic and $\divisor(q)=\divisor(f)$. By
Abel--Jacobi a degree-zero divisor $\sum_i\mu_ic_i$ is principal if and only if
\begin{equation}\label{eq:aj}
  \textstyle\sum_i \mu_i c_i=0 \quad\text{in the group law of } E .
\end{equation}
By Lemma~\ref{lem:d4}, $q$ is primitive if and only if for every $e>1$ dividing $k$ with
$e\mid\mu_i$ for all $i$,
\begin{equation}\label{eq:prim}
  \textstyle\sum_i (\mu_i/e)c_i\ne0 .
\end{equation}

\emph{Necessity.} If $n_0=0$ then $f$ has no zeros or poles, hence is constant, and
$q=c(dz)^k=(c^{1/k}dz)^k$ is a $k$-th power; for $k>1$ this contradicts primitivity, while
for $k=1$ it is the flat torus with marked points. If $n_0=2$ the nonzero orders are
$(\mu,-\mu)$ at distinct $c_1\ne c_2$, and~\eqref{eq:aj} reads $\mu(c_1-c_2)=0$ with
$c_1-c_2\ne0$, that is $E$ has a nonzero $\mu$-torsion point; for $|\mu|=1$ there is none.

\emph{Sufficiency.} Let $e^*$ be the largest divisor of $k$ dividing every $\mu_i$, put
$\nu=\mu/e^*$, and pick $T\in E$ of exact order $e^*$ (so $T=0$ if $e^*=1$). Since the
divisors of $k\in\{1,2,4\}$ form a chain, every $e>1$ occurring in~\eqref{eq:prim} divides
$e^*$. It therefore suffices to find \emph{distinct} points with $\sum_i\nu_ic_i=T$: then
$\sum_i\mu_ic_i=e^*T=0$, giving~\eqref{eq:aj}, while for each relevant $e$,
$\sum_i(\mu_i/e)c_i=(e^*/e)T$ has exact order $e$ and so is nonzero, giving~\eqref{eq:prim}.

If $n_0=2$ then $\nu=(\nu,-\nu)$ and we need $\nu(c_1-c_2)=T$. For $e^*>1$, multiplication
by $\nu$ is surjective, so choose $u$ with $\nu u=T$; then $u\ne0$ as $T\ne0$. For $e^*=1$
we have $T=0$ and $|\nu|=|\mu|\ge2$, so take $u$ a nonzero $|\nu|$-torsion point, which
exists as $E[|\nu|]\cong(\ZZ/|\nu|)^2$. Put $c_1=u$, $c_2=0$.

If $n_0\ge3$, let $\Psi\colon E^{n_0}\to E$ be $\Psi(c)=\sum_i\nu_ic_i$ over the nonzero
indices. Some $\nu_i\ne0$ and multiplication by it is surjective, so $\Psi$ is; the identity
component $K^0$ of $\ker\Psi$ is an abelian subvariety of dimension $n_0-1\ge2$, and
$\Psi^{-1}(T)$ contains the irreducible positive-dimensional coset $x_0+K^0$. For $i\ne j$
among the nonzero indices, $c_i-c_j$ is non-constant on $K^0$: if $\nu_i+\nu_j\ne0$ the
one-parameter subgroup $(c_i,c_j)=(\nu_jt,-\nu_it)$ lies in $K^0$ and there
$c_i-c_j=(\nu_i+\nu_j)t$; otherwise take a third nonzero index $l$, available since
$n_0\ge3$, and $(c_i,c_l)=(\nu_lt,-\nu_it)$, on which $c_i-c_j=\nu_lt$. Translation by $x_0$
preserves non-constancy, so each diagonal $\{c_i=c_j\}$ meets $x_0+K^0$ in a proper closed
subset; by irreducibility their finite union is not everything, so some point of
$\Psi^{-1}(T)$ has distinct coordinates. The marked points do not appear in $\Psi$ and go
anywhere else on $E$.
\end{proof}

\begin{corollary}[The torus]\label{cor:torus}
A holonomy signature on the torus is realizable if and only if it is \emph{not} one of
\begin{enumerate}[label=\emph{(\arabic*)},itemsep=1pt]
  \item no cones, with $\im\rho\ne0$;
  \item exactly two cones, of angles $3\pi/2$ and $5\pi/2$;
  \item exactly two cones, of angles $\pi$ and $3\pi$, with $\im\rho=2\Zf$.
\end{enumerate}
\end{corollary}

\begin{proof}
Apply Theorem~\ref{thm:main} with $g=1$, $k=4/d$, $\mu=m/d$ and read
Proposition~\ref{prop:genus1}. The empty cases are: $n_0=0$ with $k>1$, that is no cones and
$d\ne4$, which is (1); and $n_0=2$ with $|\mu|=1$, that is two cones with orders $(m,-m)$
and $|m|=d$. Since $d\in\{1,2,4\}$ and $m>-4$, the possibilities are $|m|=d=1$, giving cone
angles $3\pi/2$ and $5\pi/2$, and then $D=\Zf$ forces $d=1$, so the case is independent of
$\rho$, which is (2); and $|m|=d=2$, giving cone angles $\pi$ and $3\pi$ with
$\im\rho=2\Zf$, which is (3). The value $|m|=d=4$ would require an order $-4$, which is
excluded.
\end{proof}

Case (2) is the exception of~\cite{shen2022} and the non-existence of a
$3,5$-quadrangulation of the torus of~\cite{ikrss2013}; here it comes out in one line, as
the absence of a nonzero $1$-torsion point. Case (3) is the empty stratum $Q(1,-1)$
of~\cite{masur1993}. It is worth contrasting (3) with the \emph{realizable} signature ``two
cones of angles $\pi$ and $3\pi$, some odd holonomy'': the cone angles alone do not decide.

\subsection{The five exceptions}\label{sec:five}

\begin{theorem}\label{thm:classification}
A holonomy signature $(g,m,\rho)$ is realizable if and only if it is not one of the five
families of Table~\ref{tab:exceptions}. For $g\ge1$ these five are precisely the
exceptional list of (S7), read through Theorem~\ref{thm:main}; for $g=0$ there are none,
by Proposition~\ref{prop:genus0}.
\end{theorem}

\begin{figure}[t]
\centering
\begin{tikzpicture}[font=\footnotesize,>=Stealth,
  cone/.style={circle,fill=black,inner sep=1.7pt}]

\newcommand{\gone}{%
  \draw[thick,fill=black!4,rounded corners=13pt] (-1.15,-0.62) rectangle (1.15,0.62);
  \draw[thick,fill=white] (0,0) ellipse (0.30 and 0.18);}
\newcommand{\gtwo}{%
  \draw[thick,fill=black!4,rounded corners=13pt] (-1.5,-0.62) rectangle (1.5,0.62);
  \draw[thick,fill=white] (-0.68,0) ellipse (0.28 and 0.17);
  \draw[thick,fill=white] ( 0.68,0) ellipse (0.28 and 0.17);}

\begin{scope}
  \node at (0,1.30) {family 1};  \gone
  \node[align=center,font=\scriptsize] at (0,-1.12) {no cones\\ $\im\rho\ne0$};
\end{scope}
\begin{scope}[xshift=3.6cm]
  \node at (0,1.30) {family 2};  \gone
  \node[cone] at (-0.62,0.34) {};  \node[cone] at (0.62,0.34) {};
  \node[align=center,font=\scriptsize] at (0,-1.12) {$3\pi/2$ and $5\pi/2$\\ any $\rho$};
\end{scope}
\begin{scope}[xshift=7.2cm]
  \node at (0,1.30) {family 3};  \gone
  \node[cone] at (-0.62,0.34) {};  \node[cone] at (0.62,0.34) {};
  \node[align=center,font=\scriptsize] at (0,-1.12) {$\pi$ and $3\pi$\\ $\im\rho=2\Zf$};
\end{scope}

\begin{scope}[xshift=1.3cm,yshift=-3.3cm]
  \node at (0,1.30) {family 4};  \gtwo
  \node[cone] at (0,0.36) {};
  \node[align=center,font=\scriptsize] at (0,-1.12) {one cone of $6\pi$\\ $\im\rho=2\Zf$};
\end{scope}
\begin{scope}[xshift=5.9cm,yshift=-3.3cm]
  \node at (0,1.30) {family 5};  \gtwo
  \node[cone] at (-0.55,0.36) {};  \node[cone] at (0.55,0.36) {};
  \node[align=center,font=\scriptsize] at (0,-1.12) {$3\pi$ and $5\pi$\\ $\im\rho=2\Zf$};
\end{scope}
\end{tikzpicture}
\caption{The five unrealizable holonomy signatures of Table~\ref{tab:exceptions}; dots mark
cones. Families 1 to 3 live on the torus and are proved in
Proposition~\ref{prop:genus1}; family~2 is the exception already known
to~\cite{shen2022,ikrss2013}. Families 4 and 5 appear to be new in this setting: a
genus-two surface whose cross field splits globally into two line fields, the middle column
of Figure~\ref{fig:three}, cannot be quadrangulated with a single cone of angle $6\pi$, or
with a pair of cones of angles $3\pi$ and $5\pi$. Every other Gauss--Bonnet-admissible
signature, in every genus, is realizable.}
\label{fig:exceptions}
\end{figure}
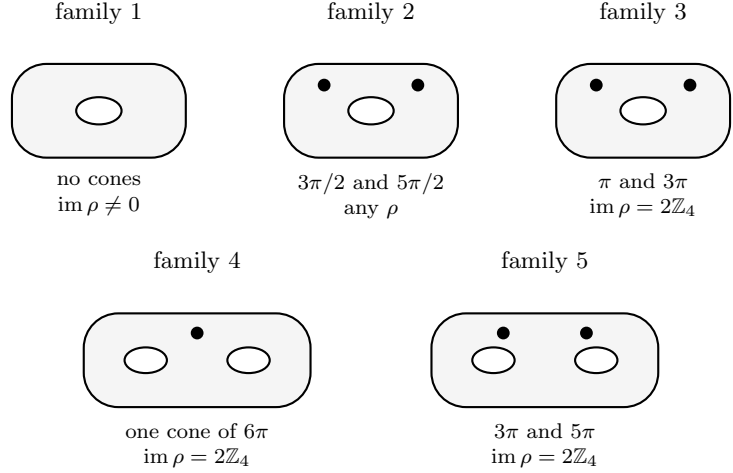

\begin{table}[t]
\centering
\begin{tabular}{@{}clllll@{}}
\toprule
 & genus & cone angles & holonomy & reduced stratum & status \\
\midrule
1 & $1$ & none & $\im\rho\ne0$ & none & Prop.~\ref{prop:genus1} \\
2 & $1$ & $3\pi/2,\ 5\pi/2$ & any & primitive $4$-diff. & Prop.~\ref{prop:genus1} \\
3 & $1$ & $\pi,\ 3\pi$ & $\im\rho=2\Zf$ & $Q(1,-1)$ & Prop.~\ref{prop:genus1} \\
4 & $2$ & $6\pi$ & $\im\rho=2\Zf$ & $Q(4)$ & (S7)(iii) \\
5 & $2$ & $3\pi,\ 5\pi$ & $\im\rho=2\Zf$ & $Q(1,3)$ & (S7)(iii) \\
\bottomrule
\end{tabular}
\caption{The unrealizable holonomy signatures.}
\label{tab:exceptions}
\end{table}

\begin{proof}
For $g=0$, Proposition~\ref{prop:genus0} says every Gauss--Bonnet-admissible signature is
realizable, and Table~\ref{tab:exceptions} lists none; (S7) is neither available nor needed
there. Assume $g\ge1$ from now on.

By Theorem~\ref{thm:main}, $s$ is realizable exactly when the stratum of primitive
$k$-differentials with orders $\mu=(m_1/d,\dots,m_n/d)$ is non-empty, where $k=4/d$. The
hypotheses of (S7) are met: $\sum_i m_i/d=4(2g-2)/d=k(2g-2)$, and $m_i>-4$ with $d\mid m_i$
gives $m_i/d>-4/d=-k$. So the unrealizable signatures are exactly the exceptional list of
(S7), pulled back along $\mu=m/d$. We unwind its three cases.

\emph{(i) $g=1$, $\mu=(1,-1)$, any $k$.} Then $m=(d,-d)$, cone angles $(d+4)\pi/2$ and
$(4-d)\pi/2$. For $d=1$: angles $5\pi/2$ and $3\pi/2$, family~2, the exception
of~\cite{shen2022,ikrss2013}. For $d=2$: angles $3\pi$ and $\pi$ with $\im\rho=2\Zf$,
family~3. For $d=4$ the datum $m=(4,-4)$ violates $m_i>-4$ and does not arise.

\emph{(ii) $g=1$, $\mu=\emptyset$, $k\ge2$.} No cones, and $k\ge2$ says $d\in\{1,2\}$, that
is $\im\rho\ne0$, family~1. For $k=1$, that is $d=4$, the stratum is the flat torus and is
non-empty; this is why family~1 carries the hypothesis $\im\rho\ne0$.

\emph{(iii) $g=2$, $k=2$, $\mu=(4)$ or $(3,1)$.} Here $d=2$, so $\im\rho=2\Zf$ and $m=2\mu$.
From $\mu=(4)$: $m=(8)$, one cone of angle $6\pi$, family~4. From $\mu=(3,1)$: $m=(6,2)$,
cones of angles $5\pi$ and $3\pi$, family~5.

Those are the five rows of Table~\ref{tab:exceptions}, and for $g\ge1$ (S7) says there is
nothing else, for any $d$.
\end{proof}

The five are drawn in Figure~\ref{fig:exceptions}. Families 4 and 5 appear to be new in this
setting. In the language of cross fields: \emph{a genus-two surface carrying a cross field
whose holonomy along every homology loop is a multiple of $\pi$, equivalently one that
splits globally into a pair of line fields, cannot be quadrangulated if its singularities
are a single cone of angle $6\pi$, or a pair of cones of angles $3\pi$ and $5\pi$.} Both
have even cone orders, so they lie outside the gcd condition of~\cite{shen2022} and
contradict nothing in it; they populate the gap it left, and \S\ref{sec:examples} works
through that gap.

\section{Explicit square-tiled realizations}\label{sec:construction}

Theorem~\ref{thm:classification} settles non-emptiness by citation. This section does
something a citation does not: it gives an explicit constructive procedure in the
square-tiled model of Remark~\ref{rem:mesh}, where a mesh is its own certificate, and it is
where every explicit witness in this paper comes from.

\subsection{Two local surgeries}\label{sec:surgeries}

\begin{lemma}[Splitting Lemma]\label{lem:split}
Let $M$ be a closed quad mesh, $v$ a vertex of valence $w$, and $e$ a \emph{loop} at
$v$, that is an edge both of whose endpoints are $v$. Let $\gamma$ be the angular gap
of $e$: the number of corners at $v$ strictly between the two ends of $e$, counted
counterclockwise from the outgoing end. Cut $M$ along $e$ and glue one unit square
into the slit as follows (Figure~\ref{fig:split}). In the dart model of
Remark~\ref{rem:mesh}, let $x$ be a dart of $e$ and $\alpha(x)$ its partner, let the new
square have darts $y_0,\dots,y_3$ with $y_s$ its side $s$, and set
\[
  \alpha'(x)=y_2,\qquad \alpha'(\alpha(x))=y_0,\qquad \alpha'(y_1)=y_3,
\]
leaving $\alpha$ unchanged elsewhere. The two halves of the cut edge are thus attached
to \emph{opposite} sides of the new square, which is the half-turn and is what makes the
move split rather than merely subdivide, while the two leftover sides $y_1,y_3$, also
opposite, are glued \emph{to each other by a translation}. The result $M'$ is a closed
quad mesh with
\begin{enumerate}[label=\emph{(\roman*)},itemsep=1pt]
  \item one more square;
  \item the same genus;
  \item the same vertices, except that $v$ is replaced by two vertices of valences
  $\gamma+2$ and $w-\gamma+2$;
  \item the same $\im\rho$.
\end{enumerate}
\end{lemma}

\begin{proof}
\emph{Cutting.} $e$ is a geodesic segment from $v$ to itself, so cutting along it opens a
slit: the two copies of $e$ become boundary, and the cyclic order of the $w$ corners at $v$
is severed at the two ends of $e$, leaving two boundary arcs with $\gamma$ and $w-\gamma$
corners.

\emph{Filling.} Sides $y_2,y_0$ of the new square are glued to the two copies of $e$, both
unit geodesics, and attaching them to opposite sides is the half-turn; then $y_1,y_3$ are
glued to each other, leaving no boundary. That last gluing is a \emph{translation}: by
Remark~\ref{rem:mesh} its rotation is $(3-1+2)\equiv0$. This is what the lemma turns on and
is not optional; gluing $y_1$ to $y_3$ by a half-turn would put $2$ into $\im\rho'$ and
destroy~(iv) whenever $\im\rho=0$.

\emph{Vertices.} The square contributes two corners to each end of the slit, giving valences
$\gamma+2$ and $w-\gamma+2$. No other vertex is touched, since the surgery happens inside
the closed star of $v$. This is where the hypothesis that $e$ is a \emph{loop} is used, and
it cannot be dropped: cutting an edge with an endpoint $u\ne v$ adds corners at $u$.

\emph{Genus.} $F$ and $V$ each rise by one and $E$ by two, the cut edge being replaced by
the two sides glued to the square together with the self-glued edge; so $\chi$ is unchanged.

\emph{Holonomy.} Old dual cycles survive, so $\im\rho\subseteq\im\rho'$, and their values
survive too: if the old gluing joined side $s$ to side $t$, with rotation
$r_{\mathrm{old}}=(t-s+2)$, the two new edges carry $r_1=(2-s+2)\equiv-s$ and
$r_2=(t-0+2)\equiv t+2$, summing to $r_{\mathrm{old}}$. Conversely the new dual graph is the
old one with that edge subdivided, which does not change the cycle space, plus one loop at
the new node, dual to $\{y_1,y_3\}$ and of rotation $0$. Hence $\im\rho'=\im\rho$.
\end{proof}

Part~(iv) is the delicate one when $\im\rho=0$: there the leftover gluing is the only place
a rotation could enter, and it does not.

So one loop of gap $\gamma$ realizes the split $(\gamma+2,\,w+2-\gamma)$, and as $\gamma$
runs over $1,\dots,w-1$ the smaller part runs over $3,\dots,w+1$. The extreme split
producing valence $1$ comes from varying the surgery rather than the loop, and it behaves
differently enough to deserve its own statement.

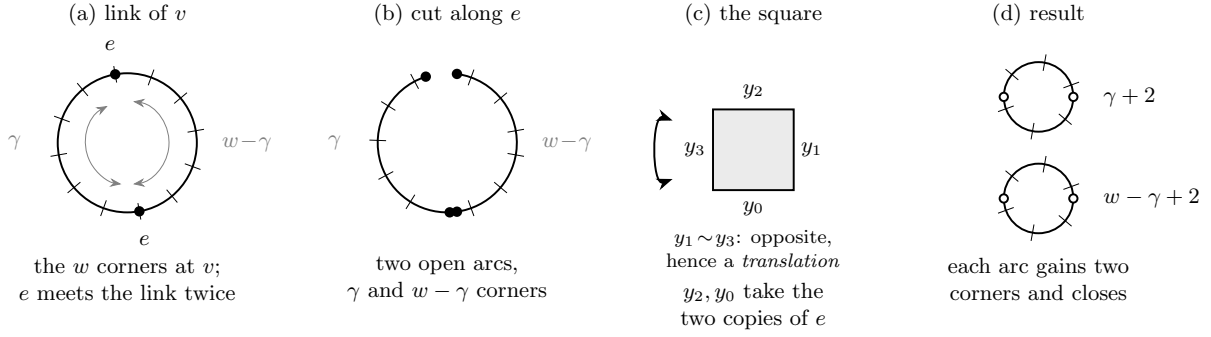
\begin{figure}[t]
\centering
\resizebox{\textwidth}{!}{%
\begin{tikzpicture}[scale=1.0,>=Stealth,font=\footnotesize,
  cor/.style={draw,thick},
  end/.style={circle,fill=black,inner sep=1.5pt},
  new/.style={circle,draw,thick,fill=white,inner sep=1.2pt}]

\begin{scope}
  \node at (0,1.85) {(a) link of $v$};
  \draw[cor] (0,0) circle (1.0);
  \foreach \a in {10,40,...,350}{\draw (\a:0.88) -- (\a:1.12);}
  \node[end] at (100:1.0) {}; \node[end] at (280:1.0) {};
  \node at (100:1.42) {$e$}; \node at (280:1.42) {$e$};
  \draw[<->,gray] (118:0.60) arc (118:262:0.60);
  \draw[<->,gray] (82:0.60) arc (82:-82:0.60);
  \node[gray] at (-1.62,0) {$\gamma$};
  \node[gray] at (1.72,0) {$w\!-\!\gamma$};
  \node[align=center] at (0,-1.95) {the $w$ corners at $v$;\\ $e$ meets the link twice};
\end{scope}

\begin{scope}[xshift=4.6cm]
  \node at (0,1.85) {(b) cut along $e$};
  \draw[cor] (108:1.0) arc (108:272:1.0);
  \draw[cor] (82:1.0) arc (82:-82:1.0);
  \foreach \a in {118,148,...,262}{\draw (\a:0.88) -- (\a:1.12);}
  \foreach \a in {10,40,70,-10,-40,-70}{\draw (\a:0.88) -- (\a:1.12);}
  \node[end] at (108:1.0) {}; \node[end] at (272:1.0) {};
  \node[end] at (82:1.0) {};  \node[end] at (-82:1.0) {};
  \node[gray] at (-1.62,0) {$\gamma$};
  \node[gray] at (1.72,0) {$w\!-\!\gamma$};
  \node[align=center] at (0,-1.95) {two open arcs,\\ $\gamma$ and $w-\gamma$ corners};
\end{scope}

\begin{scope}[xshift=9.0cm,yshift=-0.1cm]
  \node at (0,1.95) {(c) the square};
  \draw[thick,fill=black!8] (-0.58,-0.58) rectangle (0.58,0.58);
  \node at (0,0.84) {$y_2$};  \node at (0,-0.84) {$y_0$};
  \node at (-0.84,0) {$y_3$}; \node at (0.84,0) {$y_1$};
  \draw[<->,thick] (-1.16,0.42) to[out=160,in=200] (-1.16,-0.42);
  \node[align=center,font=\scriptsize] at (0,-1.42)
    {$y_1\!\sim\! y_3$: opposite,\\ hence a \emph{translation}};
  \node[align=center] at (0,-2.25) {$y_2,y_0$ take the\\ two copies of $e$};
\end{scope}

\begin{scope}[xshift=13.1cm]
  \node at (0,1.85) {(d) result};
  \draw[cor] (0,0.66) circle (0.50);
  \foreach \a in {20,80,...,320}{\draw ([shift={(0,0.66)}]\a:0.38) -- ([shift={(0,0.66)}]\a:0.62);}
  \node[new] at ([shift={(0,0.66)}]0:0.50) {};
  \node[new] at ([shift={(0,0.66)}]180:0.50) {};
  \node[anchor=west] at (0.80,0.66) {$\gamma+2$};
  \draw[cor] (0,-0.80) circle (0.50);
  \foreach \a in {20,80,...,320}{\draw ([shift={(0,-0.80)}]\a:0.38) -- ([shift={(0,-0.80)}]\a:0.62);}
  \node[new] at ([shift={(0,-0.80)}]0:0.50) {};
  \node[new] at ([shift={(0,-0.80)}]180:0.50) {};
  \node[anchor=west] at (0.80,-0.80) {$w-\gamma+2$};
  \node[align=center] at (0,-1.95) {each arc gains two\\ corners and closes};
\end{scope}
\end{tikzpicture}}
\caption{The splitting surgery of Lemma~\ref{lem:split}, seen in the link of the cone.
(a) The $w$ corners at $v$ form a circle, which the loop $e$ meets twice, separating
$\gamma$ corners from $w-\gamma$. (b) Cutting along $e$ opens the circle into two arcs.
(c) One square is glued in. Its \emph{opposite} sides $y_2,y_0$ take the two copies of $e$,
which is the half-turn, and its two remaining sides $y_1,y_3$, also opposite, are glued to
each other. The rotation of that last gluing is $(3-1+2)\equiv0$, a translation; by a
half-turn it would be $2$, which would enter $\im\rho$ and break part~(iv). (d) Each arc
receives two of the square's four corners, drawn hollow, and closes up, giving cones of
valence $\gamma+2$ and $w-\gamma+2$.}
\label{fig:split}
\end{figure}

\begin{lemma}[Odd twist]\label{lem:oddtwist}
In the situation of Lemma~\ref{lem:split}, set instead
\[
  \alpha'(x)=y_1,\qquad \alpha'(\alpha(x))=y_0,\qquad \alpha'(y_2)=y_3 ,
\]
so that the two halves of the cut edge go to \emph{adjacent} sides of the new square and
the leftover sides $y_2,y_3$, also adjacent, are glued to each other. Then $M'$ is a
closed quad mesh with one more square, the same genus, every other vertex untouched, and
$v$ replaced by two vertices of valences
\[
  1 \quad\text{and}\quad w+3 ,
\]
\emph{independently of the gap $\gamma$}. Its holonomy is $\im\rho'=\Zf$, whatever
$\im\rho$ was.
\end{lemma}

\begin{proof}
Cutting is as before. For the filling, the square's four corners are $c_0,\dots,c_3$ with
$c_t$ between sides $y_{t-1}$ and $y_t$. Gluing $y_2$ to $y_3$ folds the corner $c_3$
between them onto itself, so $c_3$ becomes a vertex with exactly one corner: valence $1$.
The same identification joins the far end of one slit arc to the far end of the other, so
the two arcs, carrying $\gamma$ and $w-\gamma$ corners, are no longer separated; together
with the three remaining corners $c_0,c_1,c_2$ they form a single vertex of
$\gamma+(w-\gamma)+3=w+3$ corners. This is why $\gamma$ drops out, and Euler's count is as
before: $F$ and $V$ rise by one, $E$ by two.

For the holonomy, the rotation across the leftover edge is $(3-2+2)\equiv3\pmod4$, an odd
number, because $y_2$ and $y_3$ are \emph{adjacent} rather than opposite. That edge is dual
to a loop of the dual graph, so $3\in\im\rho'$, and a subgroup of $\Zf$ containing an odd
element is $\Zf$.
\end{proof}

So the odd twist is available exactly when $d=1$, and that is exactly when it is needed,
since valence $1$ means order $-3$, which is odd, and $d\mid m_i$ then forces $d=1$. The
constraint is automatic, not a side condition to check.

For valence $2$ the algorithm uses a third move: cut \emph{two} edges and rejoin the eight
loose sides by one of the $105$ pairings. We claim no lemma for it. It is not a subdivision,
since the old dual cycles need not survive, so it can lower $\im\rho$ as well as raise it,
and different pairings give different valence multisets. Its outputs are admitted only when
Theorem~\ref{thm:construction} applies to them, which is why the construction needs no
general statement about the move itself.

\subsection{Base meshes}\label{sec:base}

\begin{lemma}[Base meshes]\label{lem:base}
For every $g\ge2$ the following gluing of $N=2g-1$ unit squares is a closed quad mesh
of genus $g$ with a \emph{single} vertex, of valence $8g-4$, and $\im\rho=\Zf$.
Numbering the darts $0,\dots,4N-1$ so that dart $4f+s$ is side $s$ of square $f$, as in
Figure~\ref{fig:basemesh} for $g=2$, glue
\[
\begin{array}{l@{\qquad}l}
  (0,2),\ (1,4),\ (3,5) & \text{\emph{first cap}}\\
  (s,s+2),\ (s+1,s+4),\ (s+3,s+6),\ (s+5,s+7) & \text{\emph{blocks}},\ s=6+8b,\
    b=0,\dots,g-3\\
  (t,t+2),\ (t+1,t+4),\ (t+3,t+5) & \text{\emph{last cap}},\ t=4N-6.
\end{array}
\]
Replacing this by the staircase pattern
\[
  (0,2),\qquad (2i-1,\,2i+2)\ \ \text{for } i=1,\dots,2N-2,\qquad (4N-3,\,4N-1)
\]
gives, for the same $N$, a one-vertex mesh of genus $g$ with \emph{trivial} holonomy:
a translation surface in the minimal abelian stratum $\mathcal H(2g-2)$.
\end{lemma}

\begin{proof}
Each pattern uses every dart exactly once, so it defines a closed mesh, with $2N=4g-2$
edges, matching $3+4(g-2)+3$ in the first case.

\emph{One vertex.} This says $\nu=\alpha\circ\sigma^{-1}$ is a single $4N$-cycle; we induct
on $g$. For $g=2$, $\nu$ runs $0,5,1,2,4,10,11,7,8,9,6,3$ and closes: one $12$-cycle.

For the step put $t=4N-6$. Below $4N$ the genus-$(g+1)$ pattern differs from the genus-$g$
one in a single pair: the cap's $(t+3,t+5)$ becomes $(t+3,t+6)$ and $(t+5,t+7)$, while
$(t,t+2)$ and $(t+1,t+4)$ survive as the first two pairs of the new block. Since
$\sigma^{-1}(t+4)=t+3$ and $\sigma^{-1}(t+2)=t+5$, and no other dart maps into
$\{t+3,t+5\}$, exactly two values of $\nu$ change,
\[
  \nu(t+4):\ t+5\mapsto t+6,\qquad \nu(t+2):\ t+3\mapsto t+7,
\]
and on the eight new darts the new block and cap give
\[
  t+6\mapsto t+12\mapsto t+13\mapsto t+9\mapsto t+10\mapsto t+11\mapsto t+8\mapsto t+5,
  \qquad t+7\mapsto t+3 .
\]
So the old cycle's step $t+4\to t+5$ is replaced by an arc through seven new darts and its
step $t+2\to t+3$ by the single new dart $t+7$. Replacing two steps of a cycle by two
disjoint arcs covering all eight new darts leaves one cycle, of length $4N+8$.

The staircase is the same argument: with $M=4N$, the two changed values are at $M-4$ and
$M-2$, and the steps $M-4\to M-3$ and $M-2\to M-1$ become
$M-4\to M+2\to M+4\to M+5\to M+1\to M-3$ and $M-2\to M\to M+6\to M+7\to M+3\to M-1$.

Euler's formula then gives $\chi=1-2N+N=1-N=2-2g$.

The holonomy is not forced by the cone, whose order $8g-8$ is divisible by $4$. Both claims
about $\im\rho$ have closed-form proofs, valid for every $g$.

For the staircase, assign to square $f$ the chart rotation $r_f=f\bmod4$; we check that $r$
trivialises every transition, that is that $r_{h}-r_{f}$ equals the rotation across each
glued pair, which forces $\im\rho=0$. The pair $(0,2)$ joins sides $0$ and $2$ of square $0$
to itself: rotation $(2-0+2)\equiv0$, and $r_0-r_0=0$. For a middle pair $(2i-1,2i+2)$,
write $i=2j$ or $i=2j+1$. If $i=2j$ the darts are $4(j-1)+3$ and $4j+2$, so the rotation is
$(2-3+2)\equiv1$ and $r_j-r_{j-1}=1$. If $i=2j+1$ they are $4j+1$ and $4(j+1)+0$, so the
rotation is $(0-1+2)\equiv1$ and $r_{j+1}-r_j=1$. The last pair $(4N-3,4N-1)$ joins sides
$1$ and $3$ of square $N-1$: rotation $(3-1+2)\equiv0$, and the difference is $0$. Every
edge is accounted for, so $\rho$ is the coboundary of $r$ and vanishes on every cycle.

For the first pattern it suffices to exhibit one dual cycle of odd rotation, since a
subgroup of $\Zf$ containing an odd element is $\Zf$. The first cap glues $(1,4)$ and
$(3,5)$; darts $1,3$ lie in square $0$ and darts $4,5$ in square $1$, so these two edges
form a cycle of the dual graph. Its rotation is $(0-1+2)+(3-1+2)\equiv1+0=1$, odd. The first
cap is present for every $g\ge2$, so $\im\rho=\Zf$ throughout.
\end{proof}

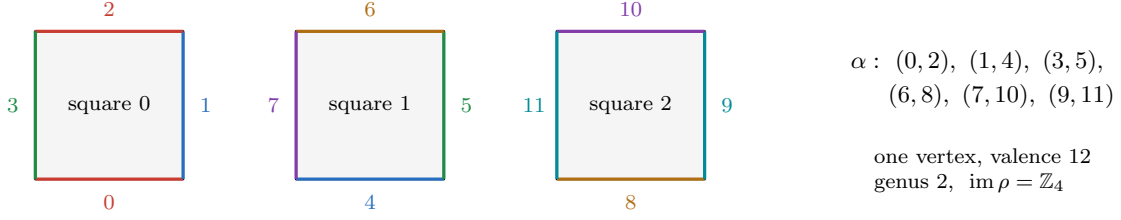
\begin{figure}[t]
\centering
\begin{tikzpicture}[font=\footnotesize,scale=1.15,
  sq/.style={draw,thick,fill=black!4},
  lb/.style={font=\scriptsize}]

\definecolor{pA}{RGB}{200,60,50}
\definecolor{pB}{RGB}{40,110,190}
\definecolor{pC}{RGB}{30,140,70}
\definecolor{pD}{RGB}{175,110,20}
\definecolor{pE}{RGB}{130,60,170}
\definecolor{pF}{RGB}{0,140,150}
\foreach \f/\x/\b/\r/\t/\l/\cb/\cr/\ct/\cl in
  {0/0/0/1/2/3/pA/pB/pA/pC,
   1/3.0/4/5/6/7/pB/pC/pD/pE,
   2/6.0/8/9/10/11/pD/pF/pE/pF}{
  \begin{scope}[xshift=\x cm]
    \draw[sq] (0,0) rectangle (1.7,1.7);
    \draw[very thick,\cb] (0,0) -- (1.7,0);
    \draw[very thick,\cr] (1.7,0) -- (1.7,1.7);
    \draw[very thick,\ct] (1.7,1.7) -- (0,1.7);
    \draw[very thick,\cl] (0,1.7) -- (0,0);
    \node[lb] at (0.85,0.85) {square $\f$};
    \node[lb,\cb] at (0.85,-0.26) {$\b$};
    \node[lb,\cr] at (1.96,0.85) {$\r$};
    \node[lb,\ct] at (0.85,1.96) {$\t$};
    \node[lb,\cl] at (-0.26,0.85) {$\l$};
  \end{scope}}

\node[align=left] at (10.9,1.15)
  {$\alpha:\ (0,2),\ (1,4),\ (3,5),$\\[2pt] $\phantom{\alpha:\ }(6,8),\ (7,10),\ (9,11)$};
\node[align=left,font=\scriptsize] at (10.9,0.10)
  {one vertex, valence $12$\\ genus $2$,\ \ $\im\rho=\Zf$};
\end{tikzpicture}
\caption{The base mesh of Lemma~\ref{lem:base} at $g=2$: $N=2g-1=3$ unit squares, side $s$
of square $f$ carrying dart $4f+s$, numbered counterclockwise from the bottom. The six
gluings, one colour each, pair every dart once, so the mesh is closed. Here
$\nu=\alpha\circ\sigma^{-1}$ is a single $12$-cycle, so there is exactly one vertex, of
valence $8g-4=12$, a cone of angle $6\pi$ and order $8$. Euler's formula gives
$\chi=1-6+3=-2$, so $g=2$. The edges $(1,4)$ and $(3,5)$ both join squares $0$ and $1$ and
form a dual cycle of rotation $1+0=1$, which is odd and forces $\im\rho=\Zf$. This pair is
the first cap of the general pattern and is present for every $g\ge2$.}
\label{fig:basemesh}
\end{figure}

\begin{remark}[the parity constraint]\label{rem:parity}
Searching for base meshes with $\im\rho\subseteq2\Zf$ is far easier than it looks, because
of a structural observation. The rotation across an edge joining side $s$ to side $t$ is
$(t-s+2)\bmod4$, which is even exactly when $s\equiv t\pmod2$. So if every gluing joins
sides of equal parity, every rotation is even and hence $\im\rho\subseteq2\Zf$. Conversely,
if $\im\rho\subseteq2\Zf$ then the parities of the rotations form a $\ZZ/2$-coboundary, so
after rotating some charts by a quarter turn, which changes $\rho$ by a coboundary and not
$\im\rho$, every gluing does join sides of equal parity. Restricting a search to
parity-preserving involutions therefore loses nothing and cuts the space in half twice over;
it is how the $d=2$ bases below were found. Note also that $\sigma$ reverses parity while
$\alpha$ preserves it, so $\nu$ reverses it and every vertex of such a mesh has even
valence, as it must, since $2\mid m_i$.
\end{remark}

\subsection{The construction and what its output certifies}\label{sec:algorithm}

\begin{algorithm}\label{alg:construction}
Let $g\ge2$, $d\in\{1,2,4\}$, and let $m_1,\dots,m_n$ be integers with $m_i>-4$,
$d\mid m_i$ and $\sum_im_i=4(2g-2)$.
\begin{enumerate}[label=\emph{\arabic*.},itemsep=1pt]
  \item Start from a base mesh with $\im\rho=d\Zf$: Lemma~\ref{lem:base} for $d=1$ and
  $d=4$, for every $g\ge2$. For $d=2$ we have no closed-form family; base meshes are
  recorded for $2\le g\le8$, found by search under Remark~\ref{rem:parity}, so the
  construction as it stands runs in those genera when $d=2$. For $3\le g\le8$ the recorded
  base is the minimal one, with $2g-1$ squares and a single cone of valence $8g-4$, the
  stratum $Q(4g-4)$; only $g=2$ needs two cones, for the reason of Remark~\ref{rem:sharp}:
  $Q(4)$ is empty, so $Q(2,2)$ is used.
  \item Repeatedly pick a cone of the current mesh and split it by one square-insertion
  surgery (Lemma~\ref{lem:split}, Lemma~\ref{lem:oddtwist}, or the figure-eight move), so
  that one of the two pieces is a cone the target asks for.
  \item Backtrack over which cone to split, which target valence to peel off it, and which
  surgery to use.
\end{enumerate}
\end{algorithm}

\begin{theorem}[Certificate theorem]\label{thm:construction}
Let $A$ be any gluing array on $2g-2+n$ squares, in particular one output by
Algorithm~\ref{alg:construction}. Suppose that, read off from $A$ alone, it defines a
closed connected mesh whose vertex valences are $m_i+4$ and whose dual holonomy has image
$d\Zf$. Then $A$ is a seamless parametrization with signature $(m,\rho)$ for some $\rho$
with $\im\rho=d\Zf$; equivalently, a $4$-differential $\eta^{d}$ with $\eta$ primitive of
order $4/d$. Its stratum is therefore non-empty, and $2g-2+n$ is the smallest number of
squares the corner count allows.
\end{theorem}

\begin{proof}
The first statement is Remark~\ref{rem:mesh}: a closed quad mesh is a seamless
parametrization, a vertex of valence $v$ is a cone of order $v-4$, and $\im\rho$ is what the
dual holonomy computes, all of which is determined by $A$ alone.

For minimality, let $M'$ be \emph{any} closed quad mesh realizing the signature. Besides the
$n$ prescribed cones it may carry $r\ge0$ further vertices, which are then regular, of
valence $4$. Every corner belongs to exactly one vertex, so by Gauss--Bonnet
\[
  4F=\sum_i(m_i+4)+4r=4(2g-2)+4n+4r ,
\]
whence $F=2g-2+n+r\ge2g-2+n$. Such an $A$ has exactly $2g-2+n$ squares, so it attains the
bound.
\end{proof}

The point of stating it this way is that the conclusion depends on the \emph{output}, not on
the moves that produced it. Nothing is assumed about which surgeries the algorithm used,
including the figure-eight move, for which no valence formula is claimed, because every
invariant is re-derived from the final dart pairing.

\begin{remark}[what is and is not proved here]\label{rem:status}
Whenever the algorithm returns, it returns an explicit gluing, and
Theorem~\ref{thm:construction} turns that gluing into a complete proof for its signature,
resting on nothing unproved. Three things are \emph{not} proved. Step 2 may in principle get
stuck; ruling that out needs a statement of the form ``a cone of valence $w$ on a genus
$\ge2$ mesh always carries a loop of the required gap, possibly after remeshing'', which the
backtracking search sidesteps in practice but which we have not established. The
figure-eight move, the only local move without a general lemma, has no valence formula,
proved or conjectured; its outputs are admitted only because their invariants are verified
directly. And for $d=2$ base meshes are available only for $2\le g\le8$, by search rather
than in closed form. What the section establishes is therefore an unconditional correctness
statement for whatever the algorithm outputs, together with non-emptiness for the data it
has been run on, not a universally quantified existence theorem.

Since splitting only ever raises the number of cones, targets with no more cones than the
base are out of reach by surgery; those are handled instead by an exhaustive search over all
gluings of the minimal size. A positive result there is a minimal witness; a negative one
rules out only realizations of minimal size, since a larger mesh with extra regular vertices
could still exist. Global non-realizability comes from Theorem~\ref{thm:classification} and
(S7), never from a bounded search.
\end{remark}

\begin{remark}[a sharp check]\label{rem:sharp}
Run over the genus-two quadratic strata, that is $d=2$ and $g=2$, the construction builds
every stratum except $Q(4)$ and $Q(1,3)$, on which it fails. These are exactly the two
that~\cite{masur1993} classifies as empty. In genus $2$ the base mesh for $d=2$ cannot be
the minimal stratum for the same reason: $Q(4)$ is empty, so $Q(2,2)$ is used instead.
\end{remark}

\section{Signatures with non-coprime cone orders}\label{sec:examples}

Theorem~\ref{thm:classification} answers every admissible signature, so in particular it
answers the ones no earlier criterion could. This section pins down that region and works
through it.

\subsection{Where the gcd condition stops}\label{sec:stops}

The hypothesis of~\cite[Prop.~2]{shen2022} is $\gcd(I_1,\dots,I_n)=\tfrac14$ on the cone
indices, the loop holonomies not entering, and with $I_i=-m_i/4$ this reads
\[
  \gcd\nolimits_{\ZZ}(m_1,\dots,m_n)=1 .
\]
It is what makes their rerouting able to hit prescribed turning numbers exactly, and under
it every signature is realizable but the torus with cones $3\pi/2,5\pi/2$. Three
consequences fix the region it leaves open.

\emph{It fails whenever every cone angle is a multiple of $\pi$}, since then every $m_i$ is
even. This is the case~\cite{shen2022} itself singles out, ``indices restricted to multiples
of $\tfrac12$'', called there a realistic scenario, and it is where families 3, 4 and 5 live.

\emph{It fails whenever there is only one cone}, since then $\gcd_{\ZZ}(m)=|m_1|=4(2g-2)$,
which is $1$ for no $g$. A single-cone signature is never covered, whatever its angle.

\emph{It is strictly stronger than $\im\rho=\Zf$.} One odd $m_i$ gives $\gen{m_i}=\Zf$ and
so $\im\rho=\Zf$, but not conversely: at $g=4$ the angles $7\pi/2$ and $25\pi/2$ give
$m=(3,21)$ with $\gcd_{\ZZ}=3$. The difference is exactly the one noted after
Corollary~\ref{cor:odd}, turning numbers in $\tfrac14\ZZ$ against a holonomy class in $\Zf$,
and it means the uncovered region is larger than the middle and right columns of
Figure~\ref{fig:three}: it is all of $\{\gcd_{\ZZ}(m_i)\ne1\}$.

That region is not small. A cross field made of two globally distinguishable line fields has
$\im\rho\subseteq2\Zf$ automatically, since transport around a loop preserves the labelling
of the two families and so turns the cross by $0$ or by $\pi$; then every $m_i$ is even and
the condition cannot fire. Fields of that kind are ordinary: the principal directions away
from the umbilics, where $\kappa_1>\kappa_2$ singles out one of the two lines; anything built
from a quadratic differential; any stripe or ridge-valley pattern. A singularity of index
$i$ appears there as a cone of angle $2\pi(1-i)$, so umbilics of index $+\tfrac12$ and
$-\tfrac12$ are cones of angle $\pi$ and $3\pi$. What~\cite{shen2022} observes is a different
and compatible thing: a cross field \emph{optimized} for smoothness or curvature alignment is
a genuine $4$-symmetry field and almost always acquires an index-$\pm\tfrac14$ singularity,
which restores the condition.

\subsection{What the classification says there}\label{sec:evenregime}

\begin{corollary}\label{cor:evenregime}
Let $s=(g,m,\rho)$ be a Gauss--Bonnet-admissible signature whose cone orders are not
coprime, that is $\gcd_{\ZZ}(m_1,\dots,m_n)\ne1$, so that~\cite[Prop.~2]{shen2022} does
not apply. Then $s$ is realizable except for exactly four signatures:
\[
  \begin{array}{ll}
    g=1,\ \text{no cones},\ \im\rho\ne0; & g=1,\ \text{angles } \pi,\,3\pi,\ \im\rho=2\Zf;\\[2pt]
    g=2,\ \text{angle } 6\pi,\ \im\rho=2\Zf; & g=2,\ \text{angles } 3\pi,\,5\pi,\ \im\rho=2\Zf.
  \end{array}
\]
In particular, on that region: if some $m_i$ is odd, $s$ is realizable; and if $\im\rho=0$,
that is if the cross field lifts to a global vector field, $s$ is realizable, in every genus.
\end{corollary}

\begin{proof}
By Theorem~\ref{thm:classification} the unrealizable signatures are the five families of
Table~\ref{tab:exceptions}. Family~2 has $m=(1,-1)$, so $\gcd_{\ZZ}=1$ and it is excluded by
hypothesis; it is the exception~\cite{shen2022} already records. The other four have
$m=\emptyset$, $(2,-2)$, $(8)$ and $(6,2)$, of $\gcd_{\ZZ}$ equal to $0$, $2$, $8$ and $2$,
so all four lie in the region and are the only failures there.

For the two special cases: none of the four has an odd $m_i$, which gives the first; and
$\im\rho=0$ forces $4\mid m_i$, which none of the four satisfies. The second also follows
directly from Theorem~\ref{thm:main}, since $d=4$ makes $k=1$, and $m_i\ge0$ then rules out
all three cases of (S7), so no abelian stratum in play is empty.
\end{proof}

So four of the five unrealizable families lie outside the gcd condition, and the fifth is
the exception it already records. Everything genuinely new in
Theorem~\ref{thm:classification} is in the region this corollary describes.

\begin{table}[t]
\centering
\begin{tabular}{@{}cl@{\quad}c@{\qquad}ccc@{}}
\toprule
 &  & & \multicolumn{3}{c}{$\im\rho$}\\
\cmidrule(l){4-6}
genus & cone angles & $\gcd_{\ZZ}(m_i)$ & $\Zf$ & $2\Zf$ & $0$\\
\midrule
$1$ & none                        & $0$  & $\times$\,(1) & $\times$\,(1) & $\checkmark$\,1\\
$1$ & $\pi,\ 3\pi$                & $2$  & $\checkmark$\,2 & $\times$\,(3) & n/a\\
$1$ & $\pi,\ \pi,\ 3\pi,\ 3\pi$   & $2$  & $\checkmark$\,4 & $\checkmark$\,4 & n/a\\
$2$ & $6\pi$                      & $8$  & $\checkmark$\,3 & $\times$\,(4) & $\checkmark$\,3\\
$2$ & $3\pi,\ 5\pi$               & $2$  & $\checkmark$\,4 & $\times$\,(5) & n/a\\
$2$ & $3\pi,\ 3\pi,\ 3\pi,\ 3\pi$ & $2$  & $\checkmark$\,6 & $\checkmark$\,6 & n/a\\
$3$ & $10\pi$                     & $16$ & $\checkmark$\,5 & $\checkmark$\,5 & $\checkmark$\,5\\
$4$ & $7\pi/2,\ 25\pi/2$          & $3$  & $\checkmark$\,8 & n/a & n/a\\
\bottomrule
\end{tabular}
\caption{Eight signatures against the three possible holonomies. \emph{No row has
$\gcd_{\ZZ}(m_i)=1$, so none of them is covered by}~\cite[Prop.~2]{shen2022}.
``$\checkmark\,N$'' means realizable, $N$ being the smallest number of squares any
realization can have (Theorem~\ref{thm:construction}); ``$\times\,(k)$'' means unrealizable,
$k$ being the family of Table~\ref{tab:exceptions} responsible; ``n/a'' means the holonomy
is inadmissible for those angles, since $d$ must divide every $m_i$. The rows with mixed
verdicts are the point: the surface and the cone angles are fixed, and $\im\rho$ alone
decides. The last row has an odd cone order, so $\im\rho=\Zf$ is forced
(Corollary~\ref{cor:odd}) and the answer is easy, but the gcd is $3$, not $1$.}
\label{tab:examples}
\end{table}

\subsection{Examples}\label{sec:sixexamples}

\emph{1. A torus with two umbilics.} Cones of angles $\pi$ and $3\pi$, in the
curvature-aligned reading umbilics of index $+\tfrac12$ and $-\tfrac12$. With $\im\rho=\Zf$
there is a realization, and the smallest has two squares: Figure~\ref{fig:torus}(a). With
$\im\rho=2\Zf$ there is none, at any size. The reduced stratum is $Q(1,-1)$ on the torus,
empty by Proposition~\ref{prop:genus1}, or by (S7)(i) with $k=2$, and this is family~3. So a
torus whose principal directions have exactly two umbilics, one of each sign, admits no
quadrangulation aligned to them. It is the exact analogue, one rung down, of the
$3\pi/2$--$5\pi/2$ torus of~\cite{shen2022,ikrss2013}; it is not that exception, since here
every angle is a multiple of $\pi$, so the gcd condition does not apply.

\emph{2. Two more cones remove the obstruction.} Keep the torus and $\im\rho=2\Zf$, but ask
for angles $\pi,\pi,3\pi,3\pi$, two umbilics of each sign. Now there is a realization, and
Figure~\ref{fig:torus}(b) is one, on the four squares that Theorem~\ref{thm:construction}
shows to be the minimum. So the obstruction of Example~1 is not about the torus, and not
about the angles: it is about that pair of angles together with that holonomy. This is the
kind of statement a classification can make and a sufficient condition cannot.

\emph{3. One cone, three answers.} Genus $2$ with a single cone of angle $6\pi$, the minimal
case, a cone of valence $12$. For $\im\rho=\Zf$ the base mesh of Lemma~\ref{lem:base}
realizes it on $2g-1=3$ squares, drawn at $g=2$ in Figure~\ref{fig:basemesh}. For
$\im\rho=0$ the staircase of the same lemma realizes it on three squares, as a translation
surface in $\mathcal H(2)$. For $\im\rho=2\Zf$ it is not realizable at all: the reduced
stratum is $Q(4)$, which is empty, and this is family~4. Same surface, same singularity,
three cross fields, and the answers are yes, no, yes. The obstruction shows up already at
minimal size, since among all gluings of three squares those of genus $2$ with a single cone
of valence $12$ realize $\im\rho=\Zf$ and $\im\rho=0$ and never $2\Zf$; but what rules it
out at \emph{every} size is Theorem~\ref{thm:classification}.

\emph{4. Genus three: the obstruction is gone.} A single cone of angle $10\pi$ on a genus-$3$
surface is realizable for all three holonomies, each on the minimal $2g-1=5$ squares:
Lemma~\ref{lem:base} for $\im\rho=\Zf$ and for $\im\rho=0$, and the recorded base mesh for
$\im\rho=2\Zf$, which is the minimal quadratic stratum $Q(8)$. Families 4 and 5 are
genus-two phenomena and not the first terms of a pattern: by (S7) no primitive quadratic
stratum in genus $\ge3$ is empty.

\emph{5. And the genus-two pair.} Cones of angles $3\pi$ and $5\pi$ with $\im\rho=2\Zf$ is
family~5, reduced stratum $Q(1,3)$: a genus-two surface whose principal directions have
exactly two umbilics, of index $-\tfrac12$ and $-\tfrac32$, admits no aligned
quadrangulation. Raise the holonomy to $\im\rho=\Zf$ and four squares suffice. Keep
$\im\rho=2\Zf$ but spread the same total curvature over four cones of angle $3\pi$, four
umbilics of index $-\tfrac12$, and six squares suffice.

\emph{6. An odd cone that the condition still misses.} At $g=4$, cones of angles $7\pi/2$
and $25\pi/2$: here $m=(3,21)$, so $\gcd_{\ZZ}=3$ and~\cite[Prop.~2]{shen2022} does not
apply, while $3$ is odd, so Corollary~\ref{cor:odd} forces $\im\rho=\Zf$ and leaves a single
orbit. Corollary~\ref{cor:evenregime} says realizable, and Algorithm~\ref{alg:construction}
returns a mesh on the minimal $2g-2+n=8$ squares. This is the easy part of the uncovered
region, with no coprimality, no cone of angle $3\pi/2$ or $5\pi/2$, and still nothing to
check, and it is uncovered only because turning numbers in $\tfrac14\ZZ$ are a finer thing
to prescribe than a holonomy class in $\Zf$.

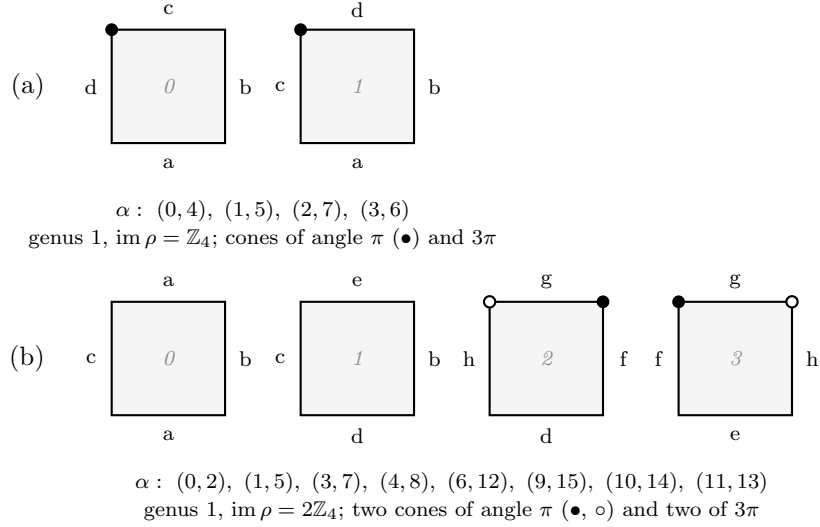
\begin{figure}[t]
\centering
\begin{tikzpicture}[font=\footnotesize,
  sq/.style={draw,thick,fill=black!4},
  fc/.style={font=\scriptsize\itshape,black!45},
  ed/.style={font=\scriptsize},
  an/.style={font=\scriptsize,align=center},
  cA/.style={circle,fill=black,inner sep=1.7pt},
  cB/.style={circle,draw,thick,fill=white,inner sep=1.5pt}]

\def\sd{1.5}

\begin{scope}
\node[anchor=east,font=\small] at (-0.75,0.75) {(a)};
\foreach \f/\x/\eb/\er/\et/\el in {0/0/a/b/c/d, 1/2.5/a/b/d/c}{
  \begin{scope}[xshift=\x cm]
    \draw[sq] (0,0) rectangle (\sd,\sd);
    \node[fc] at (0.5*\sd,0.5*\sd) {\f};
    \node[ed] at (0.5*\sd,-0.27) {\eb};
    \node[ed] at (\sd+0.27,0.5*\sd) {\er};
    \node[ed] at (0.5*\sd,\sd+0.27) {\et};
    \node[ed] at (-0.27,0.5*\sd) {\el};
    \node[cA] at (0,\sd) {};
  \end{scope}}
\node[an,anchor=north] at (2.0,-0.62)
  {$\alpha:\ (0,4),\ (1,5),\ (2,7),\ (3,6)$\\[1pt]
   genus $1$,\ $\im\rho=\Zf$;\ cones of angle $\pi$ ($\bullet$) and $3\pi$};
\end{scope}

\begin{scope}[yshift=-3.6cm]
\node[anchor=east,font=\small] at (-0.75,0.75) {(b)};
\foreach \f/\x/\eb/\er/\et/\el in
  {0/0/a/b/a/c, 1/2.5/d/b/e/c, 2/5.0/d/f/g/h, 3/7.5/e/h/g/f}{
  \begin{scope}[xshift=\x cm]
    \draw[sq] (0,0) rectangle (\sd,\sd);
    \node[fc] at (0.5*\sd,0.5*\sd) {\f};
    \node[ed] at (0.5*\sd,-0.27) {\eb};
    \node[ed] at (\sd+0.27,0.5*\sd) {\er};
    \node[ed] at (0.5*\sd,\sd+0.27) {\et};
    \node[ed] at (-0.27,0.5*\sd) {\el};
  \end{scope}}
\node[cA] at (5.0+\sd,\sd) {};   \node[cA] at (7.5,\sd) {};
\node[cB] at (5.0,\sd) {};       \node[cB] at (7.5+\sd,\sd) {};
\node[an,anchor=north] at (4.5,-0.62)
  {$\alpha:\ (0,2),\ (1,5),\ (3,7),\ (4,8),\ (6,12),\ (9,15),\ (10,14),\ (11,13)$\\[1pt]
   genus $1$,\ $\im\rho=2\Zf$;\ two cones of angle $\pi$ ($\bullet$, $\circ$) and two of $3\pi$};
\end{scope}
\end{tikzpicture}
\caption{Two explicit minimal witnesses on the torus, in the notation of
Figure~\ref{fig:basemesh}: side $s$ of square $f$ carries dart $4f+s$, numbered
counterclockwise from the bottom, and edges bearing the same letter are glued. Marked
corners are the cones of angle $\pi$, that is the vertices of valence $2$, where only two
squares meet; the remaining corners fall into the cones of angle $3\pi$. \textbf{(a)} The
signature of Example~1 with $\im\rho=\Zf$. Change the holonomy to $2\Zf$ and no such mesh
exists at any size, by family~3. \textbf{(b)} Adding one cone of each angle restores
realizability with $\im\rho=2\Zf$: the reduced stratum becomes $Q(1,1,-1,-1)$, which is
non-empty, and here is a mesh in it.}
\label{fig:torus}
\end{figure}

\section{Extensions and open problems}\label{sec:extensions}

\subsection{Surfaces with boundary: feature curves}\label{sec:boundary}

Cutting a surface along a network of feature curves produces a compact surface with boundary
on which the parametrization must send the boundary to axis-parallel segments. This
subsection carries the theory over. Let $M$ be compact oriented of genus $g$ with $b\ge1$
boundary components and interior marked points $C$. Boundary vertices are corners of angle
$\alpha=a\pi/2$ with $a\ge1$; a corner with $a=2$ is a straight point and carries no
information. The signature records the interior orders $m_i$, the cyclic sequences of corner
angles, and $\rho$, with
\[
  \rho(\partial_j)\equiv t_j:=\textstyle\sum_k(2-a_{j,k}) \pmod 4 ,
\]
the \emph{turning} of the $j$-th boundary component in quarter turns; this is forced, since
the boundary is geodesic between corners.

\begin{lemma}[Gauss--Bonnet with corners]\label{lem:gbboundary}
For a feature-aligned structure,
$\sum_i(4-v_i)+\sum_{j,k}(2-a_{j,k})=4\chi(M)=4(2-2g-b)$, where $v_i=m_i+4$.
Consequently $\sum_i m_i+\sum_{j,k}a_{j,k}$ is even.
\end{lemma}

\begin{proof}
Gauss--Bonnet for a flat metric with conical singularities and geodesic boundary with
corners reads $\sum_i(2\pi-\theta_i)+\sum_{j,k}(\pi-\alpha_{j,k})=2\pi\chi$; divide by
$\pi/2$. Reducing modulo $2$ gives the parity statement.
\end{proof}

The parity constraint is necessary but not sufficient for quad-mesh realizability. The total
corner count is $4F=\sum_i v_i+\sum_{j,k}a_{j,k}=4n+\sum_i m_i+\sum_{j,k}a_{j,k}$, so
divisibility by four asks for $\sum_i m_i+\sum_{j,k}a_{j,k}\equiv0\pmod4$, whereas
Lemma~\ref{lem:gbboundary} gives only $\equiv0\pmod2$; the remaining factor of two is the
evenness of the boundary edge count, forced separately by
$4F=2E_{\mathrm{int}}+E_\partial$. The two conditions are independent.

\begin{lemma}\label{lem:affineboundary}
$H_1(M\setminus C;\ZZ)$ is free of rank $2g+b+n-1$; the classes $\gamma_i$ and
$\partial_j$ satisfy the single relation $\sum_i\gamma_i=\sum_j\partial_j$; and
$H_1(M\setminus C)/\gen{\gamma_i,\partial_j}\cong H_1(M)/\gen{\partial_j}\cong\ZZ^{2g}$
with unimodular induced intersection form. Hence the set of signatures with
prescribed orders and turnings is a coset of $H^1(M;\Zf)\cong\Zf^{2g}$.
\end{lemma}

\begin{proof}
The first two statements are standard. Capping each boundary component with a disk gives a
closed genus-$g$ surface and identifies the quotient with its first homology, on which the
intersection form is the standard symplectic one. Two signatures with the same orders and
turnings differ by a homomorphism killing all $\gamma_i$ and $\partial_j$, that is by an
element of $\operatorname{Hom}(\ZZ^{2g},\Zf)$.
\end{proof}

\begin{theorem}[Reduction Lemma with boundary]\label{thm:reductionboundary}
Let $g\ge1$ and fix the interior orders and the corner data, and put
$D=\gen{m_1,\dots,m_n,t_1,\dots,t_b}\le\Zf$. Two feature-aligned signatures with these
data lie in the same orbit of the group of homeomorphisms preserving $C$ and each
boundary component with its corner sequence if and only if they have the same
$\im\rho$; the number of orbits is the number of subgroups of $\Zf$ containing $D$.
\end{theorem}

\begin{proof}
As for Theorem~\ref{thm:reduction}, with the three ingredients adapted.
Lemma~\ref{lem:push} applies unchanged to pushing an interior cone, and also to
\emph{sliding a boundary component}: a collar of $\partial_j$ plays the role of the
puncture, the slide along $\alpha$ is again a product $T_{\alpha_L}T_{\alpha_R}^{-1}$ of
twists along the boundary curves of an annulus containing the collar, and the same
computation gives $x\mapsto x+\langle x,\alpha\rangle[\partial_j]$, hence
$\rho\mapsto\rho+t_j\langle\cdot,\alpha\rangle$; the slide preserves the corner sequence. By
Lemma~\ref{lem:affineboundary} these moves give exactly the translations by
$D\cdot\Zf^{2g}$. Lemma~\ref{lem:sp} is unchanged. In Lemma~\ref{lem:basepoint}, replace the
disk $D_0$ by an embedded genus-$0$ subsurface $P$ containing $C$ and all $b$ boundary
components of $M$ and having one further boundary circle; then
$\Sigma=M\setminus\operatorname{int}(P)$ has genus $g$ and one boundary circle, and setting
$\rho_0=0$ on a symplectic basis inside $\Sigma$ is consistent because
$\rho_0([\partial\Sigma])=\sum_im_i-\sum_jt_j=0$ by Lemma~\ref{lem:affineboundary}. The rest
of the argument is verbatim.
\end{proof}

\begin{corollary}\label{cor:features}
If some interior order $m_i$ is odd, or some boundary component has odd turning $t_j$,
equivalently an odd number of corners of odd angle, then $D=\Zf$, a single orbit remains,
and realizability does not depend on $\rho$.
\end{corollary}

A single right-angle corner does \emph{not} by itself make $t_j$ odd: a component with one
corner of angle $\pi/2$ and one of angle $3\pi/2$ has $t_j=(2-1)+(2-3)=0$. What is needed is
an odd \emph{number} of odd-angle corners on the same component. Feature networks in
practice are full of odd-angle corners, and nothing makes their number even on every
component, so $D=\Zf$ is the normal state of affairs and the rotations along homology loops
are usually not part of the problem. This is a precise reason why the sufficient conditions
in the feature-aligned literature work as well as they do.

\begin{lemma}[Doubling]\label{lem:double}
Let $\widetilde M=M\cup_\partial\overline M$ be the double, a closed surface of genus
$2g+b-1$. A feature-aligned structure on $M$ doubles to a flat $\Zf$-cone structure on
$\widetilde M$ in which each interior cone of order $m_i$ gives two cones of order
$m_i$, each boundary corner of angle $a\pi/2$ gives one interior cone of angle $a\pi$,
that is of order $2a-4$, and
$\im\widetilde\rho\supseteq\gen{\im\rho,\ 2a\bmod 4:\text{all corners}}$.
\end{lemma}

\begin{proof}
Reflecting the atlas across each boundary arc extends it over the mirror copy: a boundary arc
develops to an axis-parallel segment, and reflection in such a segment is an isometry of the
plane, so after re-orienting the mirror charts the doubled atlas again has transitions in
$\Zf\ltimes\RR^2$. The genus is $2g+b-1$ by Euler characteristic. A boundary corner of angle
$\alpha$ becomes an interior point of total angle $2\alpha$, and an interior cone is
duplicated. Finally $\im\widetilde\rho$ contains $\im\rho$ because $H_1(M)$ maps to
$H_1(\widetilde M)$, and it contains the loops around the new cones, whose
$\widetilde\rho$-values are $2a-4\equiv2a$.
\end{proof}

The inclusion can be strict: crossing the boundary contributes a rotation by $\pi$ whenever
the two reflected charts differ in direction, so the double can carry holonomy the signature
of $M$ does not. The smallest example is a genus-$0$ surface with three straight boundary
circles and trivial $\rho$, whose double has $\im\widetilde\rho=2\Zf$.

\begin{corollary}\label{cor:doubleobstruction}
If every closed signature consisting of genus $2g+b-1$, the doubled orders of
Lemma~\ref{lem:double}, and a subgroup $S\supseteq\gen{\im\rho,2a\bmod4}$ is
unrealizable, then the feature-aligned signature is unrealizable.
\end{corollary}

This is the only direction doubling gives for free, and by itself it gives nothing. Over all
$3058$ admissible feature-aligned signatures whose corner count fits in four squares,
Corollary~\ref{cor:doubleobstruction} never applies. The reason is structural: the doubled
orders are even at the corners, so a double that could land in one of the empty quadratic
strata has $\im\widetilde\rho=2\Zf$, while the same orders with $\im\widetilde\rho=\Zf$ give
a primitive $4$-differential, and those strata are non-empty. Deciding the boundary case
therefore requires non-emptiness of \emph{real} strata, meaning $k$-differentials invariant
under a prescribed anti-holomorphic involution with the feature network as fixed locus,
which is the main question this paper leaves open on that side.

\subsection{Fixed conformal structure}\label{sec:fixed}

A second, differently quantified answer exists in the literature: at a \emph{fixed}
conformal structure and \emph{fixed} cone positions there is an Abel--Jacobi
criterion~\cite{qmg2,qmg3}. That line of work replaced the earlier quadratic-differential
(foliation) approach of~\cite{foliation2017} by meromorphic \emph{quartic} differentials
precisely in order to reach singularities of odd topological valence, and
Corollary~\ref{cor:oddvalence} below is the reason that step was necessary. The following
makes the relation to the present paper precise.

\begin{proposition}\label{prop:fixed}
Let $X$ be a Riemann surface of genus $g$, $p_1,\dots,p_n$ distinct points, and $m_i>-4$
integers with $\sum_im_i=4(2g-2)$. Put
\[
  E(X,p,m)=\Big\{e\ \Big|\ 4\ :\ e\mid m_i\ \forall i\ \text{ and }\
  \textstyle\sum_i(m_i/e)p_i\sim (4/e)K_X\Big\},
\]
where $\sim$ denotes linear equivalence. Then a $4$-differential $q$ with
$\divisor(q)=\sum_im_ip_i$ exists if and only if $1\in E$, in which case $q$ is unique
up to scale and $\im\rho_q=d\Zf$ with $d=\max E(X,p,m)$.
\end{proposition}

\begin{proof}
A meromorphic section of $K_X^4$ with prescribed divisor exists if and only if the bundle
\[
  K_X^4\otimes\mathcal O\big(-\textstyle\sum_i m_ip_i\big)
\]
is trivial, which is the displayed linear equivalence for $e=1$, and it is then unique up to
a constant. By Lemma~\ref{lem:d4} we have
$q=\eta^{e}$ for a $(4/e)$-differential $\eta$ exactly when $\im\rho_q\subseteq e\Zf$, and
such an $\eta$ has
$\divisor(\eta)=\sum(m_i/e)p_i$, so it exists if and only if $e\in E$. (Given $e\in E$ one
gets $\eta$ with $\eta^{e}=cq$ for some constant $c\ne0$; replacing $\eta$ by
$c^{-1/e}\eta$ makes $\eta^{e}=q$ exactly.) Since the subgroups of $\Zf$ are totally
ordered, $\im\rho_q$ corresponds to the largest such $e$.
\end{proof}

Thus at fixed $(X,p)$ the holonomy is not free but computed. On a torus, where $K$ is
trivial, $e\in E$ reads $\sum_i(m_i/e)c_i=0$ in the group law, which is exactly
conditions~\eqref{eq:aj} and~\eqref{eq:prim} of Proposition~\ref{prop:genus1}.

\begin{corollary}\label{cor:twoquestions}
A holonomy signature with orders $m$ and $\im\rho=d\Zf$ is realizable if and only if
there exist a Riemann surface $X$ of genus $g$ and distinct points $p_i$ with
$\max E(X,p,m)=d$.
\end{corollary}

So the fixed-structure criterion is the pointwise condition, and the question of this paper
is whether the locus it cuts out is non-empty. That is what the parametrization problem
needs, a parametrization being free to choose its flat metric.

\begin{corollary}\label{cor:oddvalence}
If some cone has odd topological valence, that is if some $m_i$ is odd, then
$E(X,p,m)\subseteq\{1\}$ for every $(X,p)$. Hence \emph{if} a $4$-differential with divisor
$\sum_im_ip_i$ exists on $(X,p)$ at all, it is primitive, with $\im\rho=\Zf$.
\end{corollary}

\begin{proof}
$e\in E$ requires $e\mid m_i$ for all $i$; an odd $m_i$ forces $e=1$.
\end{proof}

The inclusion cannot be strengthened to an equality. Whether $1\in E$, that is whether $q$
exists at all on that particular $(X,p)$, is the Abel--Jacobi condition of
Proposition~\ref{prop:fixed}, and it fails for most $(X,p)$. An odd order rules out a
nontrivial power; it does not produce a differential.

Odd valence is therefore exactly the regime $D=\Zf$ of Corollary~\ref{cor:odd}: the single
orbit, the case in which the rotations along homology loops carry no information, and the
case in which the object is a genuine $4$-differential rather than a square or a fourth
power. A method built on quadratic differentials sees only $e\ge2$ and cannot reach it; for
the topological question the same regime is the easiest one.

\subsection{Open problems}\label{sec:open}

\begin{enumerate}[label=\emph{(\arabic*)},itemsep=2pt]
  \item \emph{The external input.} Theorem~\ref{thm:classification} rests on (S7) and on
  nothing else outside \S\S\ref{sec:setup}--\ref{sec:classification}. Its genus-one cases are
  reproved here as Proposition~\ref{prop:genus1}; a self-contained proof of the genus-two
  case $k=2$, and of the assertion that the list stops there, would remove the dependence on
  (S7) and make the higher-genus non-emptiness argument self-contained.
  \item \emph{Completeness of the construction.} Three gaps, in increasing order of
  difficulty. First, prove a general lemma for the figure-eight move, including its vertex,
  genus and holonomy formulas. Cones of order $-2$ come only from that move, and it is the
  only local move the algorithm uses that still lacks a general proof:
  Lemma~\ref{lem:split} covers the generic split and Lemma~\ref{lem:oddtwist} the order $-3$
  case. Second, supply a closed-form family of base meshes with $\im\rho=2\Zf$; we have
  explicit ones only for $2\le g\le8$, found by search rather than by formula, whereas
  Lemma~\ref{lem:base} gives $d=1$ and $d=4$ in closed form for every genus. Third, and
  hardest: does step 2 of Algorithm~\ref{alg:construction} ever get stuck? Equivalently, does
  a cone of valence $w$ on a mesh of genus $\ge2$ always carry a loop of the gap the next
  split wants, after remeshing if necessary? An affirmative answer would remove the
  dependence of Theorem~\ref{thm:classification} on (S7), reproving the whole classification
  constructively and with explicit minimal witnesses.
  \item \emph{Real strata.} Deciding the feature-curve case (\S\ref{sec:boundary}) requires
  non-emptiness of strata of $k$-differentials invariant under a prescribed
  anti-holomorphic involution. All the evidence we have points at these being non-empty
  whenever Gauss--Bonnet and the parity constraint of Lemma~\ref{lem:gbboundary} allow,
  which would make feature-aligned existence unconditional.
  \item \emph{Fixed connectivity.} Everything here allows refinement (Lemma~\ref{lem:d2}).
  Existence at a fixed mesh connectivity is a different and probably much harder question.
  \item \emph{Quantization.} Square-tiled surfaces are exactly the seamless parametrizations
  that already are integer-grid maps, so the gap between Theorem~\ref{thm:main} and
  ``realizable by a mesh with at most $N$ squares'' is the quantization gap.
\end{enumerate}

\end{document}